\documentclass[12pt,reqno]{amsart}

\usepackage{amsfonts, amsmath, amssymb, amsthm}
\usepackage{anysize}
\usepackage{mathtools}
\marginsize{2cm}{2cm}{1.5cm}{1.5cm}

\usepackage[utf8]{inputenc}
\usepackage[english]{babel}
\usepackage{cmap}
\usepackage{enumerate}

\usepackage{xcolor}
\usepackage{hyperref}
\hypersetup{
	colorlinks   = true, 
	urlcolor     = blue, 
	linkcolor    = blue, 
	citecolor    = red 
}
\newcommand{\Href}[2]{\hyperref[#2]{#1~\ref{#2}}}

\theoremstyle{theorem}
\newtheorem{thm}{Theorem}[section]
\newtheorem{prp}{Proposition}[section]
\newtheorem{lem}{Lemma}[section]
\newtheorem{cor}{Corollary}[section]

\theoremstyle{definition}
\newtheorem{dfn}{Definition}[section]
\newtheorem{rem}[dfn]{Remark}

\newcommand{\st}{:\;}
\newcommand{\norm}[1]{\left\|#1\right\|}
\newcommand{\enorm}[1]{\left|#1\right|}
\newcommand{\iprod}[2]{\left\langle#1,#2\right\rangle}
\providecommand{\parenth}[1]{\left(#1\right)}
\providecommand{\braces}[1]{\left\{#1\right\}}

\newcommand{\littleo}[1]{o \! \parenth{#1}}

\newcommand{\EE}{{\mathbb E}}

\newcommand{\PP}{{\mathbb P}}
\providecommand{\abs}[1]{\lvert#1\rvert}
\providecommand{\card}[1]{\lvert#1\rvert}

\def\N{{\mathbb N}}
\def\R{{\mathbb R}}

\newcommand{\conv}{\mathrm{conv}}
\def\diam{\mathop{\rm diam}}

\newcommand{\centroid}[1]{\operatorname{c} \left(#1 \right)}
\newcommand{\ball}[1]{\mathbf{B}^{#1}}

\newcommand{\dist}[2]{\operatorname{dist}\!\left( #1, #2 \right)}

\def\phi{\varphi}
\def\epsilon{\varepsilon}

\newtheorem{modelex}[dfn]{Model example}
\providecommand{\eps}{\varepsilon}

\providecommand{\vrad}{\operatorname{Radon}}
\providecommand{\vtv}{\operatorname{Tv}}
\providecommand{\vsel}{\operatorname{Sel}}
\providecommand{\vnet}{\operatorname{WeakNet}}
\providecommand{\vhelly}{\operatorname{Helly}}
\providecommand{\ccar}{\operatorname{CCar}}
\providecommand{\KG}{\operatorname{KG}}

\providecommand{\Ip}{\operatorname{I}}

\title[Optimality of no-dimensional bounds]{Optimality of no-dimensional bounds in Banach spaces}

\author{Grigory Ivanov}
\address{Grigory~Ivanov: Pontif\'icia Universidade Cat\'olica do Rio de Janeiro\\
Departamento de Matem\'atica\\
Rua Marqu\^es de S\~ao Vicente, 225\\
Edif\'{\i}cio Cardeal Leme, sala 862\\
22451-900 G\'avea, Rio de Janeiro, Brazil}
\email{\href{mailto:grimivanov@gmail.com}{grimivanov@gmail.com}}
\thanks{G.I. is supported by Projeto Paz and Coordenacao de Aperfeicoamento de Pessoal de Nivel Superior - Brasil (CAPES) - 23038.015548/2016-06}

\author{Vladimir Kadets}
\address{Vladimir~Kadets:  School of Mathematical Sciences, Holon Institute of Technology (Israel) \\
\href{http://orcid.org/0000-0002-5606-2679}{ORCID: \texttt{0000-0002-5606-2679}}}
\email{\href{kadetsv@hit.ac.il}{kadetsv@hit.ac.il}}
\thanks{V.K. is partially supported by the KAMEA program administered by the Ministry of Absorption, Israel and by MICIU/AEI/10.13039/501100011033 and ERDF/EU through the grant PID2021-122126NB-C31}

\date{\today}

\begin{document}

\begin{abstract}
We discuss lower-bound constructions for several no-dimensional theorems of combinatorial geometry in
Banach spaces.  The common mechanism is the Maurey--Pisier theorem: the
supremal Rademacher type of a Banach space forces finite-dimensional
\(\ell_p\)-structures, and standard-coordinate configurations in these model
spaces give lower bounds for the error terms.  For the Helly approximation
property the relevant type is the type of the dual space.  For colorful Radon,
colorful Tverberg, selection, and weak \(\eps\)-net statements the relevant type
is the type of the original space.

We show that the powers appearing in the no-dimensional Helly, Radon, 
Tverberg, and selection estimates are optimal at the supremal-type exponent.
If the supremal type is attained, the known upper estimates coming from the
corresponding type inequalities have the best possible order.  We also include
endpoint statements for spaces of trivial type.  In this case the error terms
in the Helly, Radon, Tverberg, and selection statements cannot tend to zero.

Finally, we prove an endpoint obstruction for no-dimensional weak
\(\eps\)-nets in spaces of trivial type.  For every fixed cardinality bound, one
can find a finite set in the unit ball for which no approximate weak
\(\eps\)-net of that size exists below a fixed positive radius.  The proof
combines the simplex example in \(\ell_1^N\), the Lov\'asz theorem on the chromatic number of Kneser's graph, and
finite representability of \(\ell_1^N\) in spaces of trivial type.
\end{abstract}

\maketitle

\section{Introduction}

Classical results of combinatorial convexity, such as the theorems of
Carath\'eodory \cite{caratheodory1911variabilitatsbereich}, Helly
\cite{helly1923mengen}, and Tverberg \cite{tverberg1966generalization}, describe intersection and containment
properties of convex sets in a finite-dimensional vector space.
They form one of the basic parts of combinatorial convexity; see, for example,
B\'ar\'any's lectures \cite{barany2021combinatorial} and Matou\v{s}ek's book
\cite{matousek2013lectures}.  A common feature of the classical statements is
that the dimension of the ambient space enters explicitly.  In fact, these
results can be used, in different ways, to detect the dimension of the ambient
space.

A systematic study of no-dimensional, or approximate dimension-free, analogues
of these theorems was initiated by Adiprasito, B\'ar\'any, Mustafa, and Terpai
\cite{adiprasito2019theorems,adiprasito2020theorems}.  The basic principle is
to replace a dimension-dependent exact conclusion by an approximation estimate
which does not contain the dimension.  For instance, the classical
Carath\'eodory theorem says that if a point belongs to the convex hull of a set
in \(\R^d\), then it is a convex combination of at most \(d + 1\) points.  The
corresponding no-dimensional version says that if \(P\) is a bounded set in a
Euclidean space, then every point of \(\conv P\) is close to the convex hull of
at most \(k\) points of \(P\), with an error of
\[
        \frac{2\diam P}{\sqrt{k}}.
\]

This folklore statement has no dependence on the dimension of the Euclidean space, but it
only gives an approximate conclusion.  This lack of dependence on the ambient
dimension motivates the name.

There are several reasons to study such statements beyond the Euclidean setting.
First, the no-dimensional Carath\'eodory theorem is closely related to empirical
approximation by averages, a standard tool in the local theory of Banach spaces;
for instance, such approximation results are used in estimates for covering
numbers \cite[Section~5.3]{artstein2021asymptotic}, and related probabilistic
covering estimates are discussed in \cite{vershynin2018high}.  Second,
approximate Carath\'eodory-type statements also appear naturally in algorithms,
for example in Barman's work on approximate Nash equilibria and dense bipartite
subgraphs \cite{barman2015approximating}.  The Radon-type statements considered
below are connected with scale-sensitive and shattering dimensions of classes of
functions and linear functionals
\cite{alon1997scale,gurvits1997note,mendelson2004shattering}.
Further related combinatorial problems were formulated in \cite{polyanskii2024},
and links with quantum information theory were discussed in
\cite{ivanov2026nodimensionalresults}.

We now describe the properties studied in the paper. 

We start with the most basic statement, the no-dimensional Carath\'eodory
theorem.
 For a  set
\(P\subset X\), define the $k$-convex hull by
\[
        \conv_k P
        =
        \braces{
        \sum\limits_{i  =1}^s \lambda_i x_i
        \st
        1 \le s \le k,\ x_i \in P,\ \lambda_i \ge 0,
        \sum\limits_{i  =1}^s \lambda_i = 1
        }.
\]
We use $\diam P$ and $\conv P$ to denote the diameter and the convex hull of $P,$ respectively.

The no-dimensional Carath\'eodory error is
\[
        \operatorname{Car}_X(k)
        =
        \sup_P \sup_{a \in \conv P}
                \frac{\dist{a}{\conv_k P}}{\diam P},
\]
where the  supremum is taken over  sets \(P\subset X\) with positive
diameter.  We will say that  a \emph{no-dimensional Carath\'eodory theorem holds in} \(X\) if
\(\operatorname{Car}_X(k)\to 0\) as \(k \to \infty\).

In Banach space theory this statement is usually referred to as Maurey's lemma \cite{pisier1980remarques}. 

Let us recall the relevant Banach space notions.  
For \(n \in \N\), we put
\(
        [n]
        =
        \braces{1,\ldots,n}.
\)
Let
\((\epsilon_i)_{i = 1}^m\) be independent Rademacher variables.  A Banach space
\(X\) has \emph{Rademacher type} \(p \in [1,2]\) if there exists a constant
\(T_p(X) \) such that, for every finite sequence
\(x_1,\dots,x_m\in X\),
\[
        \parenth{\EE \norm{\sum\limits_{i \in [m]} \epsilon_i x_i}_X^p}^{1/p}
        \le
        T_p(X)
        \parenth{\sum\limits_{i \in [m]} \norm{x_i}_X^p}^{1/p}.
\]
We will use $T_p(X)$ to denote the least constant that appears in this inequality for $X$ of type $p.$

Every Banach space has type \(1\), and no Banach space has type strictly larger
than \(2\).  We say that \(X\) has \emph{non-trivial} type if it has type \(p\) for
some \(p > 1\).

We shall also use the weaker deterministic notion of infratype.  The space \(X\)
has \emph{infratype} \(p\) with constant \(\Ip_p(X)\) if, for every finite sequence
\(x_1,\ldots,x_m\in X\),
\[
        \min_{\theta_i=\pm1}
        \norm{\sum\limits_{i\in[m]}\theta_i x_i}_X
        \le
        \Ip_p(X)
        \parenth{\sum\limits_{i\in[m]}\norm{x_i}_X^p}^{1/p}.
\]
Type \(p\) implies infratype \(p\), with \(\Ip_p(X)\le T_p(X)\).

Maurey's lemma says that if \(X\)
has Rademacher type \(p > 1\), then
\begin{equation}
\label{eq:Car_number_Maurey_lemma_bound}
        \operatorname{Car}_X(k)
        \le
        T_p(X)\, k^{-1 + \frac{1}{p}}.
\end{equation}

To be more precise, a direct application of Maurey's lemma gives an additional factor of $2.$ To avoid misunderstandings, we  provide a short proof of this inequality in \Href{Section}{sec:caratheodory_opt}. 

The related question of when averages of finite sets admit dimension-free convexification was studied by Artstein and Kadets \cite{artstein2025b}.  Their
formulation is stated for approximation by uniform averages, rather than by
arbitrary convex combinations of at most \(k\) points, but it gives the same qualitative threshold. Namely, a 
no-dimensional Carath\'eodory theorem holds in a Banach space  if and only if the space has  non-trivial Rademacher type.  

We do not claim novelty for the 
 characterization of spaces in which the no-dimensional Carath\'eodory theorem holds; for us it serves as the first member of the
chain of no-dimensional convexity properties.

In the classical setting of combinatorial convexity, the main results are
connected by a standard chain of ideas
\[
        \text{Carath\'eodory}
        \Rightarrow
        \text{Radon}
        \Rightarrow
        \text{Tverberg}
        \Rightarrow
        \text{selection lemma}
        \Rightarrow
        \text{weak } \eps\text{-nets};
\]
see, for instance, Matou\v{s}ek's book \cite{matousek2013lectures} and  \cite{alon1992point}.  The same
strategy underlies the Euclidean no-dimensional theory of
\cite{adiprasito2019theorems,adiprasito2020theorems}.  It was later extended to
Banach spaces of non-trivial type in \cite{ivanov2021no}.  We now introduce the
corresponding no-dimensional properties in the form used in this paper.

Let \(Z_1,\ldots,Z_r\subset X\) be pairwise disjoint
sets, each of cardinality \(2n\).  Put
\[
        S = \bigcup_{j \in [r]} Z_j,
        \qquad
        D = \max_{j\in[r]} \diam Z_j.
\]
A \emph{balanced Radon split} is a partition \(S = Q_0\sqcup Q_1\) such that
\[
        \card{Q_0\cap Z_j}
        =
        \card{Q_1\cap Z_j}
        =n        
        \qquad\text{for every } j\in[r].
\]
We define the normalized Radon separation by
\[
        \vrad_X(2n,r)
        =
        \sup_{Z_1,\ldots,Z_r}
        \inf_{S = Q_0\sqcup Q_1}
        \frac{\dist{\conv Q_0}{\conv Q_1}}{D},
\]
where the infimum is over all balanced Radon splits.
We say that  a \emph{no-dimensional colorful
Radon theorem holds in} \(X\) if \(\vrad_X(2n,r)\to0\) as \(nr\to\infty\).

For the colorful Tverberg statement, let \(Z_1,\ldots,Z_r\subset X\) be pairwise
disjoint sets, each of cardinality \(k\).  A colorful \(k\)-partition is a
partition
\[
        \bigcup_{j \in [r]} Z_j
        =
        S_1\sqcup\cdots\sqcup S_k
\]
such that \(\card{S_i\cap Z_j} = 1\) for every \(i\in[k]\) and every
\(j\in[r]\).  We define the colorful Tverberg error by
\[
        \vtv_X(r,k)
        =
        \sup_{Z_1,\ldots,Z_r}
        \inf_{x,\, S = S_1\sqcup\cdots\sqcup S_k}
        \frac{\max_{i\in[k]} \dist{x}{\conv S_i}}{D},
\]
where \(D = \max_{j\in[r]}\diam Z_j\), and the infimum is over all points
\(x\in X\) and all colorful \(k\)-partitions. 
We say that a \emph{no-dimensional colorful
Tverberg theorem holds in} \(X\) if, for every \(k\ge2\),
\(
        \vtv_X(r,k) \to 0
\)
as
\(        
   r\to\infty.
\)

We shall also use a normalized form of the no-dimensional selection lemma.
For \(0 < \theta < 1\), denote by \(\vsel_X(r,\theta)\) the least number
\(\rho \ge 0\) with the following property.  For every finite set
\(P \subset X\) of size at least \(r\) and positive diameter, there is a point
\(x \in X\) such that the ball of radius \(\rho \diam(P)\)
centered at \(x\)
meets the convex hulls of at least a \(\theta\)-fraction of all \(r\)-point
subsets of \(P\).
We will say that a \emph{no-dimensional selection lemma holds in} \(X\) if
there exists a function \(0 < \theta_r < 1\) such that
\[
        \vsel_X(r,\theta_r)
        \to
        0
        \qquad\text{as } r \to \infty.
\]
The function \(\theta_r\) is part of the conclusion: it may depend on \(r\), but
not on \(P\) or on any ambient dimension.  Geometrically, this means that for
every large enough \(r\) and every finite set \(P\), one can find a point
\(x \in X\) such that a ball centered at \(x\), with radius \(\littleo{\diam P}\),
pierces the convex hulls of a prescribed positive fraction \(\theta_r\) of all \(r\)-point subsets of \(P\).

Finally, for weak \(\eps\)-nets, let
\[
        \vnet_X(\eps,M)
        =
        \sup_P
        \inf_{\substack{F \subset X\\ \card{F} \le M}}
        \sup_{\substack{Y \subset P\\ \card{Y} \ge \eps\card{P}}}
        \frac{\dist{F}{\conv Y}}{\diam P},
\]
where the first supremum is over all finite sets \(P \subset X\) with positive
diameter.  

We will say that a \emph{no-dimensional weak \(\eps\)-net theorem
holds in} \(X\) if, for every \(0 < \eps < 1\) and every \(\eta > 0\), there is
an integer \(M = M_X(\eps,\eta)\), independent of \(P\)  such that
\[
        \vnet_X(\eps,M)
        \le
        \eta .
\]

Geometrically, the set \(F\) is an approximate weak \(\eps\)-net: it is allowed
to lie anywhere in \(X\), and every subset \(Y \subset P\) with at least an
\(\eps\)-fraction of the points of \(P\) must have its convex hull close to at
least one point of \(F\).  Thus, the theorem asks for a bounded number of test
points which pierce, up to a vanishing error, all convex hulls of large subsets
of \(P\).  

It was shown in \cite{ivanov2021no}  that the  no-dimensional 
Radon, Tverberg, and weak $\epsilon$-net theorems, as well as the selection lemma hold in  Banach spaces of non-trivial type. Yet, unlike the no-dimensional Carath\'eodory theorem, no complete characterization was known for any of these results.
See the problems formulated in Section 5 of  \cite{polyanskii2024}  and Conjecture 6.3 of \cite{Barabanshchikova2026}.

We also need the Helly counterpart.  For a Banach space \(X\) and \(k \in \N\),
let \(\vhelly_X(k)\) be the infimum of all \(\alpha > 0\) with the following
property: for every finite family \(\mathcal F\) of convex subsets of
\(\ball{}_{X}\), if every subfamily of \(\mathcal F\) of size at most \(k\) has
non-empty intersection, then there is a point \(x \in X\) such that
\[
        \dist{x}{C}
        \le
        \alpha
        \qquad\text{for every } C \in \mathcal F.
\]
Equivalently, the \(\alpha\)-neighborhoods of all members of \(\mathcal F\) have
a common point.  We will say that a \emph{no-dimensional Helly theorem holds
in} \(X\) if
\[
        \vhelly_X(k)
        \to
        0
        \qquad\text{as } k \to \infty.
\]
Geometrically, this asks whether exact intersections of all small subfamilies
force an approximate global intersection after enlarging all sets by a radius
which tends to zero with \(k\). 
Earlier no-dimensional Helly estimates were obtained for uniformly convex spaces
\cite{Ivanov2025nodimHelly}, and the full characterization was later proved in
\cite{ivanov2026banachspaceshellyapproximation}.  Namely, a Banach space \(X\)
has the Helly approximation property if and only if \(X\) has non-trivial
Rademacher type. 
Let us emphasize that the Helly assertion is dual in nature.  The properties
from Carath\'eodory through weak \(\eps\)-nets are governed by the type of the
space itself.  The Helly approximation property of a space \(Y\) is governed by
the type of \(Y^*\). This is compatible with the preceding list because non-trivial type is self-dual:  \(X^*\) has non-trivial type if and only if 
\(X\) has non-trivial type.
It follows from the results of Giesy \cite{giesy1966convexity}, who proved the equivalence of the so-called \(B\)-convexity of
 \(X\) and \(X^*\), and of Maurey and Pisier \cite{maurey1976series}, who proved the equivalence between \(B\)-convexity and non-trivial type.

The first main message of the paper is qualitative.  The following theorem
collects the positive results cited above with the counterexamples proved in the
present paper.

\begin{thm}
\label{thm:intro-qualitative}
Let \(X\) be a Banach space.  The following assertions are
equivalent:
\begin{enumerate}[(i)]
\item \(X\) has non-trivial Rademacher type;
\item \(X\) has non-trivial infratype;
\item a no-dimensional Carath\'eodory theorem holds in \(X\);
\item a no-dimensional colorful Radon theorem holds in \(X\);
\item a no-dimensional colorful Tverberg theorem holds in \(X\);
\item a no-dimensional selection lemma holds in \(X\);
\item a no-dimensional weak \(\eps\)-net theorem holds in \(X\);
\item a no-dimensional Helly theorem holds in \(X\).
\end{enumerate}
\end{thm}

We believe that the proof of the equivalence for the
no-dimensional weak \(\eps\)-net theorem  is of independent  interest.   We are unaware of a proof that does not use the topological method in combinatorics \cite{Matousek_Using_BU}.

The second main result is quantitative.  The upper bounds in Maurey's lemma \cite{pisier1980remarques}
and in the no-dimensional Carath\'eodory-type results 
\cite{ivanov2021no} have the following form: if
\(X\) has type \(p>1\), then
\[
        \operatorname{Car}_X(k)
        \le
        T_p(X)\, k^{-1 + \frac{1}{p}},
\quad
        \vrad_X(2n,r)
        \le
        T_p(X)\, (nr)^{-1 + \frac{1}{p}},
\]
\[
        \vtv_X(r,k)
        \le
         \frac{2^{\frac{1}{p}}}{1 - 2^{-1 +\frac{1}{p}}}
          T_p(X)\, r^{-1 + \frac{1}{p}},
\]
and the bound on the radius in the selection lemma is provided for $\theta_r = r^{-r}$
\[
        \vsel_X(r,r^{-r})
        \le
  \parenth{
   \frac{2^{\frac{1}{p}}}{1 - 2^{-1 +\frac{1}{p}}}
    + 1 }
          T_p(X)\,      
       r^{-1 + \frac{1}{p}}.
\]
Similarly, if \(X^*\) has type \(p>1\), then the Helly approximation theorem of
\cite{ivanov2026banachspaceshellyapproximation} gives
\[
        \vhelly_X(k) 
        \le
        6 T_p(X^*)\, k^{-1 + \frac{1}{p}}.
\]
The constants here depend on the corresponding type constants.  We prove that
the powers of \(k\), \(nr\), and \(r\) cannot be improved at the supremal-type
exponent.

The \emph{supremal type} of a Banach space is the supremum of all exponents
\(p\in[1,2]\) for which the space has Rademacher type \(p\).  The \emph{supremal
infratype} is defined analogously, with infratype in place of type.  A useful
point for the present paper is that these two suprema coincide \cite[Theorem 2.1]{maurey1976series}.

\begin{thm}[Optimality at the supremal type]
\label{thm:intro-quantitative}
Let \(X\) be an infinite-dimensional Banach space, and let \(p_X\) be its supremal
type.  Then the following lower bounds hold:
\[
        \operatorname{Car}_X(k)
        \ge
        2^{-1/p_X} k^{-1 + \frac{1}{p_X}},
\qquad
        \vrad_X(2n,r)
        \ge
        (nr)^{-1 + \frac{1}{p_X}},
        \qquad
        \vtv_X(r,k)
        \ge
        \frac{1}{2}r^{-1 + \frac{1}{p_X}},
\]
and, for every \(0 < \theta < 1\),
\[
        \vsel_X(r,\theta)
        \ge
        \frac{1}{2}r^{-1 + \frac{1}{p_X}}.
\]
Consequently, if \(X\) has type \(p_X\), then the  upper
bounds for the no-dimensional Carath\'eodory, colorful Radon, and colorful Tverberg theorems
 have the optimal order; the order of the radius in the no-dimensional selection lemma is optimal as well. 
\end{thm}

\begin{thm}
\label{thm:intro-helly-quantitative}
For Helly, the same statement holds with \(X^*\) in place of \(X\): if
\(p_{X^*}\) is the supremal type of \(X^*\), then the Helly approximation
sequence of \(X\) cannot decay faster than the power
\[
        k^{-1 + \frac{1}{p_{X^*}}}.
\]
If \(X^*\) has type \(p_{X^*}\), this matches the order of the upper bound in
the no-dimensional Helly theorem.
\end{thm}

The preceding optimality statements should be read together with the distinction between
type and infratype.  Infratype is genuinely weaker at the endpoint \(p=2\), which is the case most important
for Hilbert-type estimates: Talagrand \cite{talagrand2004type} constructed a symmetric sequence space of
infratype \(2\) which is not of type \(2\).  Thus, replacing type by infratype is not merely a change of terminology at the endpoint
\(p=2\). 
We state only the no-dimensional Radon estimate in the introduction, because
it is the cleanest new observation.  The companion Carath\'eodory-type  
bounds are obtained later in \Href{Section}{sec:infratype-positive} by combining the
Kadets--Kadets averaging lemma \cite{kadets1997series} with the
Carath\'eodory--Radon--Tverberg--selection--weak-net chain from
\cite{ivanov2021no}.

\begin{thm}[Infratype upper bound for Radon's theorem]
\label{thm:intro-infratype-radon}
Let \(X\) be a Banach space of infratype \(p>1\) with constant \(\Ip_p(X)\).  Then,  for every $n,r \in \N,$
\[
        \vrad_X(2n,r)
        \le
        \Ip_p(X)\, (nr)^{-1 + \frac{1}{p}}.
\]
\end{thm}

For the sake of completeness, we also sketch the proof of the following Helly-type bound.

\begin{thm}[Infratype upper bound for Helly's theorem]
\label{thm:intro-infratype-helly}
If \(X^*\) has infratype \(p>1\) with constant \(\Ip_p(X^*)\),
then
\[
        \vhelly_X(k)
        \le
        \frac{3}{1-2^{-1+\frac{1}{p}}}\,\Ip_p(X^*) \,
         k^{-1+\frac{1}{p}}.
\]
\end{thm}

Finally, we return the dimension to the picture in the particular cases of
\(\ell_1^n\) and \(\ell_\infty^n\).  We show that the logarithmic dependence on
the dimension in the finite-dimensional estimates is already forced by explicit
coordinate examples.

\subsection*{Paper organization.}
In \Href{Section}{sec:preliminaries}, we fix notation and recall the relevant definitions and results from Banach space theory.  In \Href{Section}{sec:caratheodory_opt}, we start with the
Carath\'eodory case and explain the coordinate obstruction behind Maurey's
lemma.  The corresponding lower bounds for colorful Radon, colorful Tverberg,
and the selection lemma are proved in \Href{Section}{sec:radon},
\Href{Section}{sec:tverberg}, and \Href{Section}{sec:selection}, respectively.
Weak \(\eps\)-nets are treated in \Href{Section}{sec:weak-net}.  The Helly lower bounds, which
are dual in nature, are proved in \Href{Section}{sec:helly}. 
In
\Href{Section}{sec:infratype-positive}, we prove the infratype Radon and Helly estimates
from the introduction and record the standard positive consequences obtained
from infratype.
 In
\Href{Section}{sec:dimension_strikes_back}, we discuss finite-dimensional
examples in \(\ell_1^n\) and \(\ell_\infty^n\), showing where the dimension has
to reappear.  Finally, in \Href{Section}{sec:intro-proofs}, we formally derive
the theorems stated in the introduction from the results proved in the body of
the paper.

\section{Notations and Banach space preliminaries}
\label{sec:preliminaries}
\subsection{Basic notation}

For \(n \in \N\), we put
\(
        [n]
        =
        \braces{1,\ldots,n}.
\)
If \(A\) is a finite set and \(t \in \N\), then
\[
        \binom{A}{t}
        =
        \braces{B \subset A \st \card{B} = t}.
\]
A family of sets is called \emph{intersecting} if any two of its elements have a non-empty intersection.

For a Banach space \(X\), we denote by
\[
        \ball{}_{X}
        =
        \braces{x \in X \st \norm{x}_X \le 1}
\]
its closed unit ball.  If \(x \in X\) and \(\alpha > 0\), then
\(x + \alpha\ball{}_{X}\) denotes the closed ball with center \(x\) and
radius \(\alpha\).  For \(f \in X^*\) and \(x \in X\), we write
\(
        \iprod{f}{x}
        =
        f(x)
\)
for the value of the functional \(f\) at the vector \(x\).

For a finite set \(S\) in a linear space, the centroid
\( \centroid{S}\) is defined by
\[
        \centroid{S}
        =
        \frac{1}{\card{S} }\sum\limits_{x \in S} x.
\]

\subsection{Type and finite representability}

Let \((\epsilon_i)_{i = 1}^m\) be independent Rademacher variables.  A Banach
space \(E\) has \emph{Rademacher type} \(p \in [1,2]\) if there is a constant
\(T_p(E) < \infty\) such that, for every finite sequence
\(u_1,\ldots,u_m \in E\),
\[
        \parenth{\EE \norm{\sum\limits_{i \in [m]} \epsilon_i u_i}_E^p}^{1/p}
        \le
        T_p(E)
        \parenth{\sum\limits_{i \in [m]} \norm{u_i}_E^p}^{1/p}.
\]
Every Banach space has type \(1\), and no Banach space has type strictly larger
than \(2\).  We denote the \emph{supremal type} of \(E\) by
\[
        p_E = \sup\braces{p \in [1,2] \st E \text{ has Rademacher type } p}.
\]
We say that \(E\) has \emph{trivial type} if \(p_E = 1\).  

We shall also need the corresponding deterministic notion.  The space \(E\) has
\emph{infratype} \(p\in[1,2]\) if there exists a constant \(\Ip_p(E)<\infty\)
such that, for every finite sequence \(u_1,\ldots,u_m\in E\),
\[
        \min_{\theta_i=\pm1}
        \norm{\sum\limits_{i\in[m]}\theta_i u_i}_E
        \le
        \Ip_p(E)
        \parenth{\sum\limits_{i\in[m]}\norm{u_i}_E^p}^{1/p}.
\]
We use \(\Ip_p(E)\) for the least admissible constant.  Every Banach space has
infratype \(1\).  Type \(p\) implies infratype \(p\), and
\(\Ip_p(E)\le T_p(E)\).
As for the converse, if $p < 2,$ infratype \(p\) implies type \(p\) \cite{talagrand1992type}. However, as shown in \cite{talagrand2004type}, there are spaces of infratype \(2\) that are not of type \(2.\)

Nevertheless, there is no difference for the supremal exponents. 
Define \emph{supremal infratype} by 
\[
        \iota_E = \sup\braces{p \in [1,2] \st E \text{ has infratype } p}.
\]
The Maurey--Pisier theorem \cite[Theorem 2.1]{maurey1976series} gives
\(
        \iota_E = p_E.
\)

Thus, \(E\) has non-trivial type if and only if it has non-trivial infratype.  

We recall the finite-representability language that will be used throughout the paper.

\begin{dfn}
Let \(X\) and \(Y\) be Banach spaces.  We say that \(Y\) is \emph{finitely
representable} in \(X\) if for every finite-dimensional subspace \(F \subset Y\)
and every \(\eps > 0\) there exists a linear isomorphism
\(
        T:F\to X
\)
such that
\[
        \norm{y}
        \le
        \norm{Ty}
        \le
        (1 + \eps)\norm{y}
        \qquad\text{for every } y\in F.
\]
\end{dfn}

We use the following form of the Maurey--Pisier theorem
\cite[Th\'eor\`eme 2.3]{maurey1976series}.

\begin{prp}
\label{prp:MP}
Let \(E\) be an infinite-dimensional Banach space and let \(p_E\) be its
supremal Rademacher type.  Then \(\ell_{p_E}\) is finitely representable in
\(E\).
\end{prp}

In particular, if \(E\) has trivial type, then \(\ell_1\) is finitely
representable in \(E\).

\subsection{A transfer lemma}
We shall use \Href{Proposition}{prp:MP} in the following way. First, we
construct a finite configuration in a finite-dimensional model space \(E\), and then
we embed this model into the ambient space 
\(X\). The separation estimates are
proved inside 
\(E\). However, after the embedding, the centers witnessing intersections are allowed to be arbitrary points of \(X\), not
necessarily points of the embedded subspace. 
The following elementary
lemma shows that a lower bound on distances in the model is still enough to exclude such exterior centers.

\begin{lem}
\label{lem:separated-transfer}
Let \(T: E \to X\) be a linear map such that
\[
        \norm{u}_E
        \le
        \norm{Tu}_X
        \qquad\text{for every } u \in E.
\]
Let \(\mathcal A\) be a finite index set, 
and let 
\(\braces{C_A}_{A \in \mathcal A}\) 
be subsets of \(E\). 
Assume that 
\(
\mathcal B \subset \mathcal A
\) 
satisfies 
\[ 
\dist{C_A}{C_B} 
\ge 
\rho
 \qquad \text{for all distinct } 
 A,B \in \mathcal B. 
\]
Then, for  any two distinct \(A\) and \(B\) in \( \mathcal B\), 
 \[ 
 \dist{T\! \parenth{C_A}}{T\! \parenth{C_B}} 
  \ge 
 \rho .
  \] 
Moreover, for every \(0 < \alpha < \frac{\rho}{2}\), the  $\alpha$-neighborhoods of the sets 
\( 
\braces{T\! \parenth{C_B} \st   B \in \mathcal B} 
\)
 are pairwise disjoint. 
\end{lem}

\begin{proof}
Let \(A,B \in \mathcal B\), 
\(A \ne B\). 
For every \(a \in C_A\) 
and \(b \in C_B\), 
we have 
\[ 
\norm{Ta - Tb}_X = 
\norm{T(a - b)}_X \ge
 \norm{a - b}_E. 
 \]
Hence,
 \[ 
 \dist{T\! \parenth{C_A}}{T\! \parenth{C_B}} \ge 
 \dist{C_A}{C_B} \ge 
 \rho .
  \] 
By the triangle inequality,  the \(\alpha\)-neighborhoods of \(T\! \parenth{C_A}\) and \(T \! \parenth{C_B}\) are disjoint whenever \(2\alpha < \rho\). The lemma follows.
\end{proof}

\section{No-dimensional Carath\'eodory-type results}
\label{sec:caratheodory_opt}

Before we proceed to the examples, let us clarify the form of the lower bounds
proved in the next few sections.  We shall refer to the results connected by the
standard chain
\[
        \text{Carath\'eodory}
        \Rightarrow
        \text{Radon}
        \Rightarrow
        \text{Tverberg}
        \Rightarrow
        \text{selection lemma}
        \Rightarrow
        \text{weak } \eps\text{-nets}
\]
as \emph{Carath\'eodory-type} results.  The positive results, namely Maurey's
lemma \cite{pisier1980remarques} and the no-dimensional Radon--Tverberg
machinery of \cite{ivanov2021no}, are  proved by approximating suitable
averages by centroids.  Our coordinate constructions give slightly stronger
obstructions: we estimate the distance to the whole relevant convex hull, not
only to the centroid appearing in the proof of the positive theorem.  Thus the
examples below show that, for these Carath\'eodory-type statements, allowing
arbitrary points of the corresponding convex hulls does not improve the order of
the best possible estimates.

\subsection{Carath\'eodory lower bounds}

Recall that the no-dimensional Carath\'eodory error is
\[
        \operatorname{Car}_X(k)
        =
        \sup_P \sup_{a \in \conv P}
                \frac{\dist{a}{\conv_k P}}{\diam P}.
\]
Thus, it measures how far \(\conv{P}\) is from 
\(\conv_k P.\)

For the sake of completeness, we now obtain \eqref{eq:Car_number_Maurey_lemma_bound} by rephrasing 
  Maurey's lemma.  
Assume $a = \lambda_1 x_1 + \dots + \lambda_m x_m,$ where 
$x_1, \dots, x_m $ are points of a bounded subset $P$ of a space $X$ of type $p$ and 
$\lambda_1,  \dots, \lambda_m$ are positive numbers summing up to one. 
Let \(Y\) be a random point of \(P\) such that
\[
        \PP(Y = x_i) = \lambda_i,
        \qquad
        \EE Y = a.
\]
Let \(Y_1,\ldots,Y_k\) and \(Y_1',\ldots,Y_k'\) be independent copies of
\(Y\).  Then, by Jensen's inequality applied conditionally on
\(Y_1,\ldots,Y_k\),
\[
        \EE \norm{\sum\limits_{i \in [k]} (Y_i - a)}_X
        \le
        \EE \norm{\sum\limits_{i \in [k]} (Y_i - Y_i')}_X .
\]
The random vector \(\sum_{i \in [k]} (Y_i - Y_i')\) is symmetric, and hence, it has
the same distribution as
\(\sum_{i \in [k]} \epsilon_i (Y_i - Y_i')\).  Therefore, by H\"older's
inequality and by the type \(p\) inequality,
\[
        \EE \norm{\sum\limits_{i \in [k]} (Y_i - a)}_X
        \le
        \parenth{\EE \norm{\sum\limits_{i \in [k]}
        \epsilon_i (Y_i - Y_i')}_X^p}^{\frac{1}{p}}
        \le
        T_p(X)\parenth{\sum\limits_{i \in [k]}
        \EE \norm{Y_i - Y_i'}_X^p}^{\frac{1}{p}} .
\]
Since \(Y_i,Y_i' \in P\), we have \(\norm{Y_i - Y_i'}_X \le \diam P\), and so
\[
        \EE \norm{\sum\limits_{i \in [k]} (Y_i - a)}_X
        \le
        T_p(X) k^{\frac{1}{p}} \diam P .
\]
Thus, some realization satisfies
\[
        \norm{\frac{1}{k}\sum\limits_{i \in [k]}Y_i - a}_X
        \le
        T_p(X) k^{-1 + \frac{1}{p}} \diam P .
\]
The point \(k^{-1}\sum_{i \in [k]}Y_i\) belongs to \(\conv_k P\), which proves
\eqref{eq:Car_number_Maurey_lemma_bound} with no extra factor \(2\).

The following coordinate example shows that this power is optimal. 

\begin{modelex}
\label{modelex:caratheodory-coordinate}
Let \(1 \le p \le 2\).  Let \(m > k\).  Consider the coordinate basis
\(
        e_1,\dots,e_m
\)
of \(\ell_p^m\), and put
\[
        P_m
        =
        \braces{e_1,\dots,e_m},
        \qquad
        a_m
        =
        \centroid{P_m}
        =
        \frac{1}{m}\sum\limits_{i \in [m]} e_i.
\]
Then
\[
        D_{\operatorname{Car}}
        :=
        \diam P_m
        =
        2^{1/p}.
\]
We shall refer to \(P_m\) as the \emph{coordinate Carath\'eodory
configuration} in \(\ell_p^m\).
\end{modelex}

\begin{lem}[The coordinate Carath\'eodory obstruction]
\label{lem:caratheodory-model}
For the configuration from
\Href{Model example}{modelex:caratheodory-coordinate}, the following estimates
hold.
If \(1<p\le2\), then, for every \(b\in \conv_k P_m\),
\[
        \norm{a_m-b}_p
        \ge
        \parenth{1-\frac{k}{m}} k^{-1 + \frac{1}{p}}.
\]
If \(p=1\), then, for every \(b\in \conv_k P_m\),
\[
        \norm{a_m-b}_1
        \ge
        2\parenth{1-\frac{k}{m}}.
\]

\noindent	       
Consequently,
\[
        \operatorname{Car}_{\ell_p^m}(k)
        \ge
        2^{-1/p}\parenth{1-\frac{k}{m}}k^{-1 + \frac{1}{p}}
        \qquad\text{for }1<p\le2,
\]
and
\[
        \operatorname{Car}_{\ell_1^m}(k)
        \ge
        1-\frac{k}{m}.
\]
\end{lem}

\begin{proof}
Let \(b\in \conv_k P_m\).  Then \(b\) has non-negative coordinates,
\(\norm{b}_1 = 1\), and its coordinate support has cardinality at most \(k\).
Denote this support by \(I\), and put \(s = \card I\).  Thus \(s\le k\).

Assume first that \(1<p\le2\).  By Jensen's inequality,
\[
        \sum\limits_{i\in I}\enorm{b_i - \frac{1}{m}}^p
        \ge
        s\parenth{\frac{1}{s}-\frac{1}{m}}^p
        =
        s^{1-p}\parenth{1-\frac{s}{m}}^p.
\]
Since \(s\le k\), we have
\[
        s^{1-p}\parenth{1-\frac{s}{m}}^p
        \ge
        k^{1-p}\parenth{1-\frac{k}{m}}^p.
\]
Therefore,
\[
        \norm{a_m-b}_p
        \ge
        \parenth{1-\frac{k}{m}}k^{-1 + \frac{1}{p}}.
\]

Now let \(p=1\).  We have
\[
        \norm{a_m-b}_1
        =
        \sum\limits_{i\in I}\enorm{b_i - \frac{1}{m}}
        +
        \sum\limits_{i\notin I}\frac{1}{m}.
\]
The first sum is at least
\[
        \enorm{\sum\limits_{i\in I} b_i - \frac{s}{m}}
        =
        1-\frac{s}{m},
\]
and the second sum equals \((m-s)/m\).  Hence,
\[
        \norm{a_m-b}_1
        \ge
        2\parenth{1-\frac{s}{m}}
        \ge
        2\parenth{1-\frac{k}{m}}.
\]
Dividing by \(D_{\operatorname{Car}} = 2^{1/p}\) gives the normalized
estimates.
\end{proof}
\begin{rem}
One can obtain a more bulky estimate 
\[ 
\norm{a_m - b}_p \ge 
\parenth{ \left(1-\frac{k}{m}\right)^p k^{1-p} + \frac{m-k}{m^p}}^\frac{1}{p},
\]
which gives continuity at the endpoint $p=1.$
\end{rem}

\begin{thm}[Optimality of Maurey's lemma]
\label{thm:caratheodory-sup-type}
Let \(X\) be an infinite-dimensional Banach space, and let \(p_X\) be the
supremal type of \(X\). 
If \(p_X > 1\), then, for every \(k \in \N\),
\[
        \operatorname{Car}_X(k)
        \ge
        2^{- \frac{1}{p_X}} k^{-1 + \frac{1}{p_X}}.
\]
If \(p_X = 1\), then, for every \(k \in \N\),
\[
        \operatorname{Car}_X(k)
        =    1.
\]

In particular, if \(X\) has type \(p_X\), then the power
\(k^{-1 + \frac{1}{p_X}}\) in Maurey's lemma is optimal.
\end{thm}

\begin{proof}
First assume that \(p_X > 1\).  
Fix \(k \in \N\), \(m > k\), and
\(\delta \in (0,1)\).  
By \Href{Proposition}{prp:MP}, there is a linear map
\(
        T \st \ell_{p_X}^m \to X
\)
such that
\[
        \norm{u}_{p_X}
        \le
        \norm{Tu}_X
        \le
        (1 + \delta) \norm{u}_{p_X}
        \qquad\text{for every }
        u \in \ell_{p_X}^m.
\]
Apply \(T\) to the coordinate Carath\'eodory configuration from
\Href{Model example}{modelex:caratheodory-coordinate}.  The image of
\(a_m\) belongs to the convex hull of the transferred set.  Moreover, every
point in the \(k\)-convex hull of the transferred set has the form \(Tb\),
where \(b \in \conv_k P_m\).  By \Href{Lemma}{lem:caratheodory-model},
\[
        \norm{Ta_m - Tb}_X
        \ge
        \norm{a_m - b}_{p_X}
        \ge
        \parenth{1-\frac{k}{m}} k^{-1 + \frac{1}{p_X}}.
\]
The diameter of the transferred set is at most
\[
        (1 + \delta) D_{\operatorname{Car}}
        =
        (1  + \delta)2^{\frac{1}{p_X}}.
\]
Therefore
\[
        \operatorname{Car}_X(k)
        \ge
        \frac{1}{1 + \delta}2^{-\frac{1}{p_X}}
        \parenth{1 - \frac{k}{m}} k^{-1 + \frac{1}{p_X}}.
\]
Letting first \(m\to\infty\) and then \(\delta\to0\), we get
\[
        \operatorname{Car}_X(k)
        \ge
        2^{- \frac{1}{p_X}} k^{-1 + \frac{1}{p_X}}.
\]

Now assume that \(p_X = 1\).
  The same argument, using the \(p = 1\) part of
\Href{Lemma}{lem:caratheodory-model}, gives
\[
        \operatorname{Car}_X(k)
        \ge
        \frac{1}{1 + \delta}\parenth{1-\frac{k}{m}}.
\]
Letting \(m \to \infty\) and then \(\delta \to 0\), we get
\[
        \operatorname{Car}_X(k)
        \ge         1.
\]
The opposite inequality \(\operatorname{Car}_X(k) \le 1\) is trivial: if
\(a \in \conv P\), then for every \(x \in P\) we have
\[
        \norm{a - x}_X
        \le
        \diam P.
\]
Thus \(\operatorname{Car}_X(k) = 1\) for all \(k\).
\end{proof}

\section{Colorful Radon lower bounds}
\label{sec:radon}

Recall that the normalized Radon separation is
\[
        \vrad_X(2n,r)
        =
        \sup_{Z_1,\ldots,Z_r}
        \inf_{S = Q_0\sqcup Q_1}
        \frac{\dist{\conv\, Q_0}{\conv\, Q_1}}{D},
\]
where the infimum is over all balanced Radon splits of 
the set \(\bigcup\limits_{j \in [r]} Z_j\), where 
\(Z_1,\ldots,Z_r\subset X\) are pairwise disjoint
subsets of $X$ of cardinality \(2n\).

\begin{modelex}[Coordinate Radon configuration]
\label{modelex:radon-coordinate}
Let \(1 \le p \le 2\).    Consider the coordinate basis
\(
        e_{j,a},
\)
\(
        j \in [r],
\)
\(
 a \in [2n],
\)
of \(\ell_p^{2nr}\).  For each \(j \in [r]\), put
\[
        Z_j
        =
        \braces{e_{j,1},\ldots,e_{j,n}}.
\]
Then
\[
        D_{\vrad}
        :=
        \max_{j \in [r]} \diam Z_j
        =
        2^{1/p}.
\]
We shall refer to the sets \(Z_1,\ldots,Z_r\) as the \emph{coordinate Radon
configuration} in \(\ell_p^{2nr}\).
\end{modelex}

\begin{lem}
\label{lem:radon-model}
For the configuration from \Href{Model example}{modelex:radon-coordinate},
every balanced Radon split
\[
        S = Q_0 \sqcup Q_1,
        \qquad
        S = \bigcup_{j \in [r]} Z_j,
\]
satisfies
\[
        \dist{\conv \, Q_0}{\conv\, Q_1}
        =
        \norm{\centroid{Q_0} - \centroid{Q_1}}_p
        =
        2^\frac{1}{p} (nr)^{-1 + \frac{1}{p}}.
\]
In particular,
\[
        \vrad_{\ell_p^{2nr}}(2n,r)
        \ge
        (nr)^{-1 + \frac{1}{p}}.
\]
\end{lem}

\begin{proof}
Let
\[
        S = Q_0 \sqcup Q_1
\]
be an arbitrary balanced Radon split.  Then,
\(
        \card{Q_0} = \card{Q_1} = {nr},
\)
and the coordinate supports of \(Q_0\) and \(Q_1\) are disjoint.  If
\(
        x \in \conv \, Q_0,
\)
and 
\(
        y \in \conv \, Q_1,
\)
then \(x\) and \(y\) have non-negative coordinates, their supports are
disjoint, and
\(
        \norm{x}_1 = \norm{y}_1 = 1.
\)
Hence,
\[
        \norm{x - y}_p^p
        =
        \norm{x}_p^p + \norm{y}_p^p.
\]
The minimum of \(\norm{x}_p\) over \(\conv\, Q_0\) is attained at the centroid of
\(Q_0\), and equals
\[
        \parenth{nr}^{\frac{1}{p} - 1}.
\]
The same holds for \(Q_1\).  Therefore,
\[
        \dist{\conv\, Q_0}{\conv\, Q_1}
        =
        \norm{\centroid{Q_0} - \centroid{Q_1}}_p
        =
        \parenth{
        2\parenth{{nr}}^{1-p}
        }^{\frac{1}{p}}
        =
        2^{\frac{1}{p}} (nr)^{-1 + \frac{1}{p}}.
\]
Thus,
\[
        \frac{\dist{\conv \, Q_0}{\conv \, Q_1}}{D_{\vrad}}
        =
        (nr)^{-1 + \frac{1}{p}}.
\]
Since the balanced split was arbitrary, the claim follows.
\end{proof}

\begin{rem}
In the preceding example, the closest points of 
\(\conv{\, Q_0}\) and \(\conv{\, Q_1}\)
are precisely the centroids of \(Q_0\) and \(Q_1\).  Moreover,
\[
        \norm{\centroid{Q_0} - \centroid S}_p
        =
        \norm{\centroid{Q_1} - \centroid S}_p
        =
        2^{-1+ \frac{1}{p}} (nr)^{-1 + \frac{1}{p}},
\]
and
\[
        \dist{\conv{\, Q_0}}{\conv{\,  Q_1}}
        =
        2^{\frac{1}{p}}(nr)^{-1 + \frac{1}{p}}.
\]
Thus, in this model example, the convex-hull separation is exactly twice the
centroid error appearing in the original Radon statement.
\end{rem}

\begin{thm}
\label{thm:radon-sup-type}
Let \(X\) be an infinite-dimensional Banach space, and let \(p_X\) be the
supremal type of \(X\).  
Then, for every  \(n\) and every \(r\),
\[
        \vrad_X(2n,r)
        \ge
        (nr)^{-1 + \frac{1}{p_X}}.
\]
  If, moreover, \(X\) has type
\(p_X\), then the power
\(
        (nr)^{-1 + \frac{1}{p_X}}
\)
in the no-dimensional colorful Radon theorem is optimal.
\end{thm}

\begin{proof}
Fix \(\delta \in (0,1)\).  By \Href{Proposition}{prp:MP}, there is a linear map
\[
        T:\ell_{p_X}^{2nr}\to X
\]
such that
\[
        \norm{u}_{p_X}
        \le
        \norm{Tu}_{X}
        \le
        (1 + \delta)\norm{u}_{p_X}
        \qquad\text{for every } u\in\ell_{p_X}^{2nr}.
\]
Apply \(T\) to the coordinate Radon configuration from
\Href{Model example}{modelex:radon-coordinate}.
Since \(T\) is injective, every balanced Radon split of the transferred
configuration is the image of a unique balanced Radon split of the model
configuration.

Let \(C_0,C_1\subset \ell_{p_X}^{2nr}\) be the two convex hulls arising from an
arbitrary balanced Radon split in the model configuration.  By
\Href{Lemma}{lem:radon-model} and by \Href{Lemma}{lem:separated-transfer}, 
\[
         \dist{T\!\parenth{C_0}}{T\!\parenth{C_1}}
        \ge
        2^{\frac{1}{p_X}}(nr)^{-1 + \frac{1}{p_X}}.
\]

Again by
\Href{Lemma}{lem:radon-model}, the transferred color-class diameters are at most
\[
        (1 + \delta) D_{\vrad},
        \qquad
        D_{\vrad} = 2^{\frac{1}{p_X}}.
\]
Therefore, by the definition of \(\vrad_X(2n,r)\),
\[
        \vrad_X(2n,r)
        \ge
        2^{\frac{1}{p_X}} \frac{(nr)^{-1 + \frac{1}{p_X}}}
        {(1 + \delta) D_{\vrad}}
        =
        \frac{1}{1 + \delta}
       (nr)^{-1 + \frac{1}{p_X}}.
\]
Letting \(\delta \to 0\) gives the stated bound.
\end{proof}

The case of trivial type follows by applying
\Href{Theorem}{thm:radon-sup-type} with \(p_X = 1\).

\begin{cor}
\label{cor:radon-trivial-type}
If \(X\) has trivial type, then, for every  \(n\) and every \(r\),
\[
        \vrad_X(2n,r)
        \ge
        1.
\]
Thus, no no-dimensional colorful Radon estimate with an error tending to zero as
\(nr \to \infty\) can hold in a space of trivial type.
\end{cor}

\section{Colorful Tverberg lower bounds}
\label{sec:tverberg}
Recall that the colorful Tverberg error is
\[
        \vtv_X(r,k)
        =
        \sup_{Z_1,\ldots,Z_r}
        \inf_{x,\, S = S_1 \sqcup \cdots \sqcup S_k}
        \frac{\max_{i \in [k]} \dist{x}{\conv S_i}}{D},
\]
where the supremum is over pairwise disjoint sets
\(Z_1,\ldots,Z_r \subset X\), each of cardinality \(k\),
\[
        S
        =
        \bigcup_{j  \in [r]} Z_j,
        \qquad
        D
        =
        \max_{j \in [r]} \diam Z_j,
\]
and the infimum is over all points \(x \in X\) and all colorful
\(k\)-partitions
\[
        S
        =
        S_1 \sqcup \cdots \sqcup S_k .
\]
Here colorful means that \(\card{S_i \cap Z_j} = 1\) for every
\(i \in [k]\) and every \(j \in [r]\).  Thus \(\vtv_X(r,k)\) is the smallest
normalized radius which, in the worst case, guarantees that the convex hulls of
the parts of a colorful partition have a common approximate intersection
point.  The coordinate example below shows
that the order of this radius cannot be improved.

\begin{modelex}[Coordinate colorful Tverberg configuration]
\label{modelex:tverberg-coordinate}
Let \(1 \le p \le 2\).  Let \(r \in \N\) and \(k \ge 2\).  Consider the
coordinate basis
\(
        e_{j,a},
\)
\(
    j \in [r],
\)
\(
 a \in [k],
\)
of \(\ell_p^{rk}\).  For each \(j \in [r]\), put
\[
        Z_j
        =
        \braces{e_{j,1},\ldots,e_{j,k}}.
\]
Then
\[
        D_{\vtv}
        :=
        \max_{j \in [r]} \diam Z_j
        =
        2^{\frac{1}{p}}.
\]
We shall refer to the sets \(Z_1,\ldots,Z_r\) as the \emph{coordinate colorful
Tverberg configuration} in \(\ell_p^{rk}\).
\end{modelex}

\begin{lem}[The coordinate colorful Tverberg obstruction]
\label{lem:tverberg-model}
For the configuration from
\Href{Model example}{modelex:tverberg-coordinate}, every colorful
\(k\)-partition
\[
        S
        =
        S_1 \sqcup \cdots \sqcup S_k,
        \qquad
        S
        =
        \bigcup_{j \in [r]} Z_j,
\]
satisfies
\[
        \dist{\conv{\, S_i}}{\conv{\, S_\ell}}
        \ge
        D_{\vtv}\, r^{-1 + \frac{1}{p}}
\]
for every pair of distinct indices 
\(i,\ell \in [k]\).  
Consequently, for every
\(x_0 \in \ell_p^{rk}\),
\[
        \max_{m \in [k]} \dist{x_0}{\conv S_m}
        \ge
        \frac{1}{2}D_{\vtv} \, r^{-1 + \frac{1}{p}}.
\]
In particular,
\[
        \vtv_{\ell_p^{rk}}(r,k)
        \ge
        \frac{1}{2}r^{-1 + \frac{1}{p}}.
\]
\end{lem}

\begin{proof}
Let
\[
        S
        =
        S_1 \sqcup \cdots \sqcup S_k
\]
be an arbitrary colorful \(k\)-partition of \(S = \bigcup\limits_{j \in [r]} Z_j\).
For each \(j \in [r]\), the partition assigns the \(k\) points of \(Z_j\) to
the \(k\) classes \(S_1,\dots,S_k\).  Hence, there is a permutation
\(\sigma_j\) of \([k]\) such that
\[
        S_i
        =
        \braces{e_{j,\sigma_j(i)} \st j \in [r]}
        \qquad\text{for every } i \in [k].
\]

Fix two distinct indices \(i,\ell \in [k]\).  Take arbitrary points
\[
        x
        =
        \sum\limits_{j \in [r]} a_j e_{j,\sigma_j(i)}
        \in
        \conv S_i,
        \qquad
        y
        =
        \sum\limits_{j \in [r]} b_j e_{j,\sigma_j(\ell)}
        \in
        \conv S_\ell,
\]
where
\[
        a_j,b_j \ge 0,
        \qquad
        \sum\limits_{j \in [r]} a_j
        =
        \sum\limits_{j \in [r]} b_j
        =
        1.
\]
Since \(\sigma_j(i) \ne \sigma_j(\ell)\) for every \(j \in [r]\), the supports
of \(x\) and \(y\) are disjoint.  Therefore
\[
        \norm{x - y}_p^p
        =
        \sum\limits_{j \in [r]} a_j^p
        +
        \sum\limits_{j \in [r]} b_j^p.
\]
By Jensen's inequality,
\[
        \sum\limits_{j \in [r]} a_j^p
        \ge
        r\parenth{\frac{1}{r}}^p
        =
        r^{1-p},
        \qquad
        \sum\limits_{j \in [r]} b_j^p
        \ge
        r\parenth{\frac{1}{r}}^p
        =
        r^{1-p}.
\]
Thus,
\[
        \norm{x - y}_p
        \ge
        \parenth{2r^{1-p}}^{\frac{1}{p}}
        =
        2^{\frac{1}{p}}r^{\frac{1}{p} - 1}
        =
        2^{\frac{1}{p}}r^{-1 + \frac{1}{p}}
        =
        D_{\vtv} \, r^{-1 + \frac{1}{p}}.
\]
Since \(x \in \conv S_i\) and \(y \in \conv S_\ell\) were arbitrary, we get
\[
        \dist{\conv\, S_i}{\conv\, S_\ell}
        \ge
         D_{\vtv} \, r^{-1 + \frac{1}{p}}.
\]

Applying \Href{Lemma}{lem:separated-transfer} to the identity map on
\(\ell_p^{rk}\) with
\(
        C_m = \conv\, S_m
\)
for 
\(
        m \in [k],
\)
and 
\(
        \rho = D_{\vtv} r^{-1 + \frac{1}{p}},
\)
we see that, for every
\[
        0 < \alpha < \frac{1}{2}D_{\vtv} r^{-1 + \frac{1}{p}}, 
\]
the \(\alpha\)-neighborhoods of the sets
\(
        \conv\, S_1,\dots,\conv\, S_k
\)
are pairwise disjoint.  Since \(k \ge 2\), no point of \(\ell_p^{rk}\) can
belong to all of these neighborhoods.  Therefore, for every
\(x_0 \in \ell_p^{rk}\),
\[
        \max_{m \in [k]} \dist{x_0}{\conv\, S_m}
        \ge
        \frac{1}{2}D_{\vtv} r^{-1 + \frac{1}{p}}.
\]
Dividing by \( D_{\vtv} \), we conclude that
\[
        \vtv_{\ell_p^{rk}}(r,k)
        \ge
        \frac{1}{2}r^{-1 + \frac{1}{p}}.
\]
\end{proof}

\begin{thm}
\label{thm:tverberg-sup-type}
Let \(X\) be an infinite-dimensional Banach space, and let \(p_X\) be the
supremal type of \(X\). 
Then, for every \(r\) and every \(k \ge 2\),
\[
        \vtv_X(r,k)
        \ge
        \frac{1}{2} r^{-1 + \frac{1}{p_X}}.
\]
  If \(X\) has type \(p_X\), then the dependence \(r^{-1 + \frac{1}{p_X}}\) in the
no-dimensional colorful Tverberg theorem is optimal, for every number
\(k \ge 2\) of parts.
\end{thm}

\begin{proof}
Fix \(\delta \in (0,1)\).  By \Href{Proposition}{prp:MP}, there is a linear map
\[
        T : \ell_{p_X}^{rk} \to X
\]
such that
\[
        \norm{u}_{p_X}
        \le
        \norm{Tu}_X
        \le
        (1 + \delta)\norm{u}_{p_X}
        \qquad\text{for every } u \in \ell_{p_X}^{rk}.
\]

Apply \(T\) to the coordinate colorful Tverberg configuration from
\Href{Model example}{modelex:tverberg-coordinate}.  Since \(T\) is injective,
every colorful \(k\)-partition of the transferred configuration is the image
under \(T\) of a colorful \(k\)-partition of the model configuration.

By \Href{Lemma}{lem:tverberg-model} and by \Href{Lemma}{lem:separated-transfer},
the \(\alpha\)-neighborhoods
of the images of the convex hulls of the parts  are pairwise disjoint 
 whenever
\[
        0 < \alpha < \frac{1}{2}D_{\vtv} \, r^{-1 + \frac{1}{p_X}},
        \quad
        \text{where} \quad
 D_{\vtv}
      =
        2^{\frac{1}{p_X}}. 
\]

On the other hand, the upper estimate for \(T\) implies that the diameters of
the transferred color classes are at most
\(
        (1 + \delta)D_{\vtv}.
\)
Therefore,
\[
        \vtv_X(r,k)
        \ge
        \frac{1}{2}\frac{1}{1 + \delta}r^{-1 + \frac{1}{p_X}}.
\]
Letting \(\delta \to 0\), we obtain
\[
        \vtv_X(r,k)
        \ge
        \frac{1}{2}r^{-1 + \frac{1}{p_X}}.
\]
\end{proof}

The case of trivial type follows by applying
\Href{Theorem}{thm:tverberg-sup-type} with \(p_X = 1\). 
\begin{cor}
\label{cor:tverberg-trivial-type}
If \(X\) has trivial type, then, for every \(r\) and every \(k \ge 2\),
\[
        \vtv_X(r,k)
        \ge
        \frac{1}{2}.
\]
Thus, no colorful no-dimensional Tverberg theorem with an error tending to zero
as \(r \to \infty\) can hold in any space of trivial type.
\end{cor}

\begin{rem}
The same block-coordinate example also applies to the version in which the
normalization is taken with respect to the diameters of the parts of the final
partition, rather than with respect to \(\max_j \diam Z_j\).  Indeed, in the
model configuration each transversal \(S_i\) has diameter \(2^{1/p}\), the same
as the color classes.  Thus the lower bound above also gives the endpoint
obstruction for the second colorful no-dimensional Tverberg problem considered
in \cite[Section~5]{polyanskii2024}.
\end{rem}

\section{Selection lemma lower bounds}
\label{sec:selection}

Recall that, for \(0 < \theta < 1\), the selection parameter \(\vsel_X(r,\theta)\)
is the least normalized radius which guarantees a point piercing the convex
hulls of a prescribed fraction of all \(r\)-element subsets.  Equivalently,
\[
        \vsel_X(r,\theta)
        =
        \sup_P
        \inf_{x \in X}
        \inf\braces{\alpha > 0 \st
        \enorm{\braces{A \in \binom{P}{r} \st
        \conv A \cap \parenth{x + \alpha \diam(P) \ball{}_{X}} \ne \emptyset}}
        \ge
        \theta \binom{\card P}{r}},
\]
where the supremum is over finite sets \(P \subset X\) with
\(\card P \ge r\) and positive diameter, 
and \( \binom{P}{r}\) denotes the set of all $r$-element subsets of $P$.  Thus, a no-dimensional selection lemma
in \(X\) asks whether one can choose numbers \(0 < \theta_r < 1\) so that
\(\vsel_X(r,\theta_r) \xrightarrow[r\to\infty]{} 0\).  The coordinate example below gives the sharp
lower bound for this radius.

\begin{modelex}[Coordinate selection configuration]
\label{modelex:selection-coordinate}
Let \(1 \le p \le 2 \)  and \(N \ge r\).  Consider the coordinate basis
\(
        e_1, \dots, e_N
\)
of \(\ell_p^N\), and put
\[
        P_N
        =
        \braces{e_1, \dots, e_N}.
\]
Then
\[
        D_{\vsel}
        :=
        \diam P_N
        =
        2^{\frac{1}{p}}.
\]
For \(A \subset [N]\) of size \( r \), put
\[
        \Delta_A
        =
        \conv \braces{e_i \st i \in A}.
\]
We shall refer to \(P_N\) as the \emph{coordinate selection configuration} in
\(\ell_p^N\).
\end{modelex}

\begin{lem}[The coordinate selection obstruction]
\label{lem:selection-model}
For the configuration from
\Href{Model example}{modelex:selection-coordinate}, if
\(A,B \subset [N]\) are disjoint and
\(\card A = \card B = r\), then
\[
        \dist{\Delta_A}{\Delta_B}
        =
        D_{\vsel}\, r^{-1 + \frac{1}{p}}.
\]
\end{lem}

\begin{proof}
Let \(A,B \subset [N]\) be disjoint \(r\)-element subsets.  If
\(u \in \Delta_A\) and \(v \in \Delta_B\), then \(u\) and \(v\) have
non-negative coordinates, disjoint supports, and
\(
        \norm{u}_1
        =
        \norm{v}_1
        =
        1.
\)
Hence,
\[
        \norm{u - v}_p^p
        =
        \norm{u}_p^p + \norm{v}_p^p.
\]
The minimum of \(\norm{u}_p\) over \(\Delta_A\) is attained at the centroid of
\(\Delta_A\), and equals
\(
        r^{-1 + \frac{1}{p}}.
\)
The same holds for \(\Delta_B\).  Therefore
\[
        \dist{\Delta_A}{\Delta_B}
        =
        \parenth{2r^{1-p}}^{1/p}
        =
        2^{\frac{1}{p}}r^{-1 + \frac{1}{p}}
        =
        D_{\vsel}\, r^{-1 + \frac{1}{p}}.
\]

\end{proof}

\begin{thm}[Optimality of the selection radius]
\label{thm:selection-sup-type}
Let \(X\) be an infinite-dimensional Banach space, and let \(p_X\) be the
supremal type of \(X\). 
Then, for every \(r\) and every \(0 < \theta < 1\),
\[
        \vsel_X(r,\theta)
        \ge
        \frac{1}{2}r^{-1 + \frac{1}{p_X}}.
\]
If  \(X\) has type \(p_X\), then the dependence
\(r^{-1 + \frac{1}{p_X}}\) in the no-dimensional selection lemma is optimal.
\end{thm}

\begin{proof}
It is enough to prove that, for every
\(0 < \gamma < \frac{1}{2}\), every \(0 < \theta < 1\), and every \(r\),
there is a finite set \(P \subset X\) such that no translate of
\[
        \gamma r^{-1 + \frac{1}{p_X}}\diam(P)\ball{}_{X}
\]
intersects the convex hulls of at least
\[
        \theta\binom{\card P}{r}
\]
distinct \(r\)-element subsets of \(P\).

Fix \(0 < \gamma < \frac{1}{2}\).  Choose \(\delta \in (0,1)\) so small that
\[
        \gamma(1 + \delta)
        <
        \frac{1}{2}.
\]
Choose \(N\) such that
\[
        N \ge 2r
        \qquad\text{and}\qquad
        \frac{r}{N} < \theta.
\]
By \Href{Proposition}{prp:MP}, there is a linear map
\[
        T:\ell_{p_X}^{N} \to X
\]
such that
\[
        \norm{u}_{p_X}
        \le
        \norm{Tu}_X
        \le
        (1 + \delta)\norm{u}_{p_X}
        \qquad\text{for every } u \in \ell_{p_X}^{N}.
\]
Apply \(T\) to the coordinate selection configuration from
\Href{Model example}{modelex:selection-coordinate}, and put
\[
        P
        =
        T\!\parenth{P_N}.
\]
The lower estimate for \(T\) implies that \(T\) is injective, and hence
\(\card P = N\).  Moreover,
\[
        \diam P
        \le
        (1 + \delta)D_{\vsel},
        \qquad
        D_{\vsel}
        =
        2^{\frac{1}{p_X}}.
\]

Assume, to the contrary, that there are a point \(x_0 \in X\) and a family
\(\mathcal A \subset \binom{[N]}{r}\) such that
\[
        \card{\mathcal A}
        \ge
        \theta\binom{N}{r}
\]
and, for every \(A \in \mathcal A\),
\[
        T\!\parenth{\Delta_A}
        \cap
        \parenth{x_0 + \gamma r^{-1 + \frac{1}{p_X}}\diam(P)\ball{}_{X}}
        \ne
        \emptyset .
\]
We claim that \(\mathcal A\) is an intersecting family.  Indeed, suppose that
\(A,B \in \mathcal A\) are disjoint.  By \Href{Lemma}{lem:selection-model} and by \Href{Lemma}{lem:separated-transfer},
\[
        \dist{T\!\parenth{\Delta_A}}{T\!\parenth{\Delta_B}}
        \ge
        D_{\vsel}\, r^{-1 + \frac{1}{p_X}}.
\]
On the other hand, both sets meet the ball with center \(x_0\) and radius
\[
        \gamma r^{-1 + \frac{1}{p_X}}\diam(P)
        \le
        \gamma(1 + \delta)D_{\vsel}\, r^{-1 + \frac{1}{p_X}}
        <
        \frac{1}{2}D_{\vsel}\, r^{-1 + \frac{1}{p_X}},
\]
which is impossible.  Thus, \(\mathcal A\) is intersecting.

By the Erd\H{o}s--Ko--Rado theorem \cite{erdos1961intersection}, since
\(N \ge 2r\), every intersecting family of \(r\)-subsets of \([N]\) has
cardinality at most
\[
        \binom{N - 1}{r - 1}
        =
        \frac{r}{N}\binom{N}{r}
        <
        \theta\binom{N}{r},
\]
contradicting the choice of \(\mathcal A\).  Therefore no such translate
exists.

Since \(\gamma < \frac{1}{2}\) was arbitrary, we get
\[
        \vsel_X(r,\theta)
        \ge
        \frac{1}{2}r^{-1 + \frac{1}{p_X}}.
\]
\end{proof}

\begin{rem}
\Href{Theorem}{thm:selection-sup-type} should not be read as a failure of the
selection lemma in spaces of non-trivial type.  If \(p_X > 1\), then
\(r^{-1 + \frac{1}{p_X}} \to 0\), and the lower bound is compatible with the positive
selection theorem.  The theorem says that no choice of fractions \(\theta_r\)
can force a radius of order \(\littleo{r^{-1 + \frac{1}{p_X}}}\).  When \(p_X = 1\), the
same lower bound is constant, and this is the obstruction to a no-dimensional
selection lemma.
\end{rem}

The case of trivial type follows by applying
\Href{Theorem}{thm:selection-sup-type} with \(p_X = 1\).  Then 
\(
        \frac{1}{2}r^{-1 + \frac{1}{p_X}}
        =
        \frac{1}{2}.
\)

\begin{cor}[Endpoint selection obstruction]
\label{cor:selection-trivial-type}
If \(X\) is a Banach space of trivial type, then, for
every \(r\) and every \(0 < \theta < 1\),
\[
        \vsel_X(r,\theta)
        \ge
        \frac{1}{2}.
\]
Consequently, there is no sequence \(0 < \theta_r < 1\) such that
\(\vsel_X(r,\theta_r) \to 0\).  Thus, no no-dimensional selection lemma can hold
in spaces of trivial type.
\end{cor}

\section{Weak \texorpdfstring{\(\eps\)}{epsilon}-net lower bounds}
\label{sec:weak-net}

Recall that the weak \(\eps\)-net error is
\[
        \vnet_X(\eps,M)
        =
        \sup_P
        \inf_{\substack{F \subset X\\ \card{F} \le M}}
        \sup_{\substack{Y \subset P\\ \card{Y} \ge \eps\card{P}}}
        \frac{\dist{F}{\conv Y}}{\diam P},
\]
where the supremum is over finite sets \(P \subset X\) with positive diameter.
Thus, \(F\) is allowed to lie anywhere in the ambient space, and
\(\vnet_X(\eps,M)\) measures how well \(M\) test points can approximate all
convex hulls of subsets of \(P\) containing at least an \(\eps\)-fraction of the
points.

The endpoint result below shows that, in spaces of trivial type, no choice of a
finite cardinality bound can force the weak-net radius to tend to zero.

\begin{modelex}[Coordinate weak-net configuration]
\label{modelex:weak-net-coordinate}
Let \(N \in \N\).  Consider the coordinate basis
\(
        e_1,\dots,e_N
\)
of \(\ell_1^N\), and put
\(
        P_N
        =
        \braces{e_1,\dots,e_N}.
\)
Then
\[
        D_{\vnet}
        :=
        \diam P_N
        =
        2.
\]
For \(A \subset [N]\), put
\[
        \Delta_A
        =
        \conv\braces{e_i \st i \in A}.
\]
We shall refer to \(P_N\) as the \emph{coordinate weak-net configuration} in
\(\ell_1^N\).
\end{modelex}

For \(N,t \in \N\), with \(N \ge 2t\), the \emph{Kneser graph}
\(\KG(N,t)\) is the graph whose vertex set is \(\binom{[N]}{t}\), and in which
\(A\) and \(B\) are adjacent if and only if they are disjoint.  Recall that a
set of vertices is independent if no two of its vertices are joined by an edge.
Thus, an independent set in \(\KG(N,t)\) is the same thing as an intersecting
family of \(t\)-subsets of \([N]\).

We shall use the following classical theorem of Lov\'asz
\cite{lovasz1978kneser}, which determines the chromatic number of the Kneser
graph.
\begin{prp}[Lov\'asz--Kneser theorem]
\label{prp:lovasz-kneser}
Let \(N,t \in \N\) and assume that \(N \ge 2t\).  Then any partition of the
vertices of \(\KG(N,t)\) into independent sets contains at least
\(
        N - 2t + 2
\)
sets.
\end{prp}

\begin{thm}[Weak-net endpoint for spaces of trivial type]
\label{thm:weak-net-trivial-type}
Let \(X\) be a Banach space of trivial type.  Then, for
every \(0 < \eps < \frac{1}{2}\) and every \(M \in \N\),
\[
        \vnet_X(\eps,M)
        \ge
        \frac{1}{2}.
\]
Consequently, in spaces of trivial type there is no dimension-free weak
\(\eps\)-net theorem whose approximation radius tends to zero.
\end{thm}

\begin{proof}
It is enough to prove that, for every \(0 < \gamma < \frac{1}{2}\), there is a
finite set \(P \subset X\) such that no set \(F \subset X\) with
\(\card{F} \le M\) is a \(\gamma\diam(P)\)-approximate weak \(\eps\)-net for
\(P\).

Fix \(0 < \gamma < \frac{1}{2}\).  Choose \(\delta \in (0,1)\) so small that
\[
        \gamma(1 + \delta)
        <
        \frac{1}{2}.
\]
Choose \(N\) so large that, for
\[
        t
        =
        \lceil \eps N \rceil,
\]
one has
\[
        t \le \frac{N}{2}
        \qquad\text{and}\qquad
        M
        <
        N - 2t + 2.
\]
Since \(X\) has trivial type, \Href{Proposition}{prp:MP} gives a linear map
\[
        T:\ell_1^N \to X
\]
such that
\[
        \norm{u}_1
        \le
        \norm{Tu}_X
        \le
        (1 + \delta)\norm{u}_1
        \qquad\text{for every } u \in \ell_1^N.
\]
Apply \(T\) to the coordinate weak-net configuration and put
\[
        P
        =
        T\!\parenth{P_N}
        =
        \braces{Te_1,\ldots,Te_N}.
\]
The lower estimate for \(T\) makes \(T\) injective, so \(\card P = N\).
Moreover,
\[
        \diam P
        \le
        (1 + \delta)D_{\vnet}
        =
        2(1 + \delta).
\]

For a point \(x \in X\), consider the family
\[
        \mathcal A_x
        =
        \braces{A \in \binom{[N]}{t} \st
        T\!\parenth{\Delta_A}
        \cap
        \parenth{x + \gamma\diam(P)\ball{}_{X}}
        \ne
        \emptyset}.
\]
We claim that \(\mathcal A_x\) is an intersecting family.  Suppose that
\(A,B \in \mathcal A_x\) are disjoint. 
Clearly, the distance between  
\({\Delta_A} \) and \({\Delta_B}\) in $\ell_1$ equals \(2.\)
Therefore, by \Href{Lemma}{lem:separated-transfer},
\[
        \dist{T\!\parenth{\Delta_A}}{T\!\parenth{\Delta_B}}
        \ge
        2.
\]
On the other hand, both sets meet the ball
\(x + \gamma\diam(P)\ball{}_{X}\).  Hence,
\[
        \dist{T\!\parenth{\Delta_A}}{T\!\parenth{\Delta_B}}
        \le
        2\gamma\diam(P)
        \le
        4\gamma(1 + \delta)
        <
        2,
\]
a contradiction.  Thus \(\mathcal A_x\) is intersecting.

Now let \(F = \braces{x_1,\ldots,x_m}\), where \(m \le M\).  The families
\(\mathcal A_{x_j}\), \(j \in [m]\), are independent sets in the Kneser
graph \(\KG(N,t)\).  By \Href{Proposition}{prp:lovasz-kneser},
\[
        m
        <
        N - 2t + 2,
\]
these independent sets cannot cover all vertices of \(\KG(N,t)\).  Hence, there is
\(A \in \binom{[N]}{t}\) which belongs to none of the families
\(\mathcal A_{x_j}\).  Put
\[
        Y
        =
        \braces{Te_i \st i \in A}
        \subset P.
\]
Then
\(
        \card{Y}
        =
        t
        \ge
        \eps N
        =
        \eps\card P,
\)
and
\(
\        \dist{F}{\conv Y}
        >
        \gamma\diam(P).
\)
Thus,
\[
        \vnet_X(\eps,M)
        \ge
        \gamma .
\]
Letting \(\gamma \to \frac{1}{2}\) proves the theorem.
\end{proof}

\section{Helly approximation lower bounds}
\label{sec:helly}
Recall that \(\vhelly_X(k)\) is the infimum of all \(\alpha > 0\) with the
following property: for every finite family \(\mathcal F\) of convex subsets of
\(\ball{}_{X}\), if every subfamily of \(\mathcal F\) of size at most \(k\) has
non-empty intersection, then the \(\alpha\)-neighborhoods of all members of
\(\mathcal F\) have a common point. 

  The role of the model space is now
played by the dual of \(\ell_p^s\).

\begin{lem}
\label{lem:lp-dual-model-helly}
Let \(1 < p \le 2\), and let \(s > k\).  Then
\[
        \vhelly_{(\ell_p^s)^*}(k)
        \ge
        \frac{(s/k)^{1 - \frac{1}{p}} - 1}{s^{1 - \frac{1}{p}} + 1}
        =
        \frac{k^{-1 + \frac{1}{p}} - s^{-1 + \frac{1}{p}}}
        {1 + s^{-1 + \frac{1}{p}}}.
\]
\end{lem}

\begin{proof}
Let \(e_1,\dots,e_s\) be the coordinate basis of \(\ell_p^s\).  We use
\(e_i\) as the corresponding coordinate functionals on \((\ell_p^s)^*\).  For
\(i \in [s]\), define
\[
        K_i
        =
        \braces{y \in \ball{}_{(\ell_p^s)^*} \st
        \iprod{e_i}{y} \ge k^{-1 + \frac{1}{p}}}.
\]
If \(J \subset [s]\) and \(\card{J} = k\), then
\[
        y_J
        =
        k^{-1 + \frac{1}{p}}
        \sum\limits_{j \in J} e_j^*
        \in
        (\ell_p^s)^*
\]
belongs to \(\ball{}_{(\ell_p^s)^*}\), and \(y_J \in K_i\) for every
\(i \in J\).  Hence, every \(k\)-subfamily of
\(\braces{K_i}_{i = 1}^s\) intersects.

Assume that the \(\alpha\)-neighborhoods of all \(K_i\)'s have a common point
\(y\).  We shall prove that
\begin{equation}
\label{eq:helly-lp-dual-alpha}
        \alpha
        \ge
        \frac{(s/k)^{1 - \frac{1}{p}} - 1}{s^{1 - \frac{1}{p}} + 1}.
\end{equation}
For each \(i\), choose \(z_i \in K_i\) such that
\(\norm{z_i - y}_{(\ell_p^s)^*} \le \alpha\).  Since
\(\norm{e_i}_{\ell_p^s} = 1\),
\[
        \iprod{e_i}{y}
        \ge
        \iprod{e_i}{z_i} - \alpha
        \ge
        k^{-1 + \frac{1}{p}} - \alpha .
\]
If \(\alpha \ge k^{-1 + \frac{1}{p}}\), then
\eqref{eq:helly-lp-dual-alpha} is already true.  Thus, we may assume that
\(\alpha < k^{-1 + \frac{1}{p}}\).  In this case all coordinates of \(y\) are
bounded below by the positive number \(k^{-1 + \frac{1}{p}} - \alpha\).  Since
\(y\) is \(\alpha\)-close to \(K_1 \subset \ball{}_{(\ell_p^s)^*}\), we have
\(\norm{y}_{(\ell_p^s)^*} \le 1 + \alpha\).  On the other hand, by duality,
\[
        \norm{y}_{(\ell_p^s)^*}
        \ge
        \frac{\iprod{\sum\limits_{i \in [s]} e_i}{y}}
        {\norm{\sum\limits_{i \in [s]} e_i}_{p}}
        \ge
        s^{1 - \frac{1}{p}}
        \parenth{k^{-1 + \frac{1}{p}} - \alpha}.
\]
Consequently,
\[
        1 + \alpha
        \ge
        (s/k)^{1 - \frac{1}{p}}
        -
        s^{1 - \frac{1}{p}} \alpha .
\]
This is equivalent to \eqref{eq:helly-lp-dual-alpha}.
The lemma follows.
\end{proof}

The next elementary lemma is the dual analogue of the transfer principle, that is \Href{Lemma}{lem:separated-transfer} used in
Radon, Tverberg, selection lemma, and weak-net constructions.

\begin{lem}[Quotients generated by subspaces of the dual]
\label{lem:dual-subspace-quotient-helly}
Let \(X\) be a Banach space, let \(Z\) be a finite-dimensional normed space, and
assume that there is a linear map
\[
        T : Z^* \to X^*
\]
such that
\[
        \norm{y}_{Z^*}
        \le
        \norm{Ty}_{X^*}
        \le
        D \norm{y}_{Z^*}
        \qquad\text{for every } y \in Z^* .
\]
Then, for every \(k \in \N\),
\[
        \vhelly_X(k)
        \ge
        \frac{1}{D}\vhelly_Z(k).
\]
\end{lem}

\begin{proof}
Put
\[
        E
        =
        T\!\parenth{Z^*}
        \subset
        X^* .
\]
Then \(E\) is finite-dimensional, and \(T\) is an isomorphism from \(Z^*\) onto
\(E\) with
\[
        \norm{T} \le D,
        \qquad
        \norm{T^{-1}} \le 1 .
\]
Consider the canonical evaluation map
\[
        Q_E : X \to E^*,
        \qquad
        Q_E x \parenth{f}
        =
        \iprod{f}{x},
        \qquad f \in E .
\]
The adjoint map \(Q_E^*:E \to X^*\) is the inclusion of \(E\) into \(X^*\), and
therefore is an isometry.  Equivalently, by Goldstine's theorem,
\(Q_E\parenth{\ball{}_{X}}\) is weak\(^*\)-dense in \(\ball{}_{E^*}\); since
\(E\) is finite-dimensional, this is norm density.  Thus \(Q_E\) is a quotient
map of norm one onto \(E^*\).

We first show the resulting monotonicity.  If \(Q:X \to W\) is a quotient map
of norm one, then
\[
        \vhelly_X(k)
        \ge
        \vhelly_W(k).
\]
Indeed, fix \(0 < \lambda < 1\), and let
\(\braces{C_i}_{i \in I}\) be a finite family of convex subsets of
\(\ball{}_{W}\) witnessing a lower bound for \(\vhelly_W(k)\).  The family
\[
        K_i
        =
        Q^{-1}\!\parenth{\lambda C_i}\cap \ball{}_{X},
        \qquad i \in I,
\]
consists of convex subsets of \(\ball{}_{X}\).  If a subfamily
\(\braces{C_i}_{i \in J}\), \(\card{J} \le k\), has a common point \(w\), then
\(\lambda w \in \lambda\ball{}_{W}\).  Since \(Q\) is a quotient map and
\(\lambda < 1\), there exists \(x \in \ball{}_{X}\) with
\(Qx = \lambda w\).  Hence the corresponding subfamily of the \(K_i\)'s also
intersects.

Moreover, if the \(\alpha\)-neighborhoods of all \(K_i\)'s have a common point
\(x \in X\), then the \(\alpha\)-neighborhoods of all \(\lambda C_i\)'s have the
common point \(Qx\), because \(\norm{Q} \le 1\).  Therefore every obstruction in
\(W\), after scaling by \(\lambda\), gives an obstruction in \(X\).  Letting
\(\lambda \to 1\) proves the monotonicity.

It remains to compare \(E^*\) with \(Z\).  Let
\[
        S
        =
        \parenth{T^{-1}}^* : Z \to E^* .
\]
Since \(Z\) is finite-dimensional, we identify \(Z\) with \(Z^{**}\).  The
operator \(S\) is an isomorphism and satisfies
\[
        \frac{1}{D}\norm{z}_{Z}
        \le
        \norm{Sz}_{E^*}
        \le
        \norm{z}_{Z}
        \qquad\text{for every } z \in Z .
\]
Thus applying \(S\) to a witnessing family in \(Z\) can decrease all normalized
Helly radii by at most the factor \(D\).  More explicitly, if the
\(\beta\)-neighborhoods of the sets \(S (C_i)\) have a common point in \(E^*\),
then applying \(S^{-1}\) gives a common point for the
\(D\beta\)-neighborhoods of the sets \(C_i\) in \(Z\).  Consequently,
\[
        \vhelly_{E^*}(k)
        \ge
        \frac{1}{D}\vhelly_Z(k).
\]
Combining this with the quotient monotonicity for \(Q_E:X \to E^*\), we obtain
\[
        \vhelly_X(k)
        \ge
        \vhelly_{E^*}(k)
        \ge
        \frac{1}{D}\vhelly_Z(k),
\]
as required.
\end{proof}

\begin{thm}
\label{thm:helly-sup-type}
Let \(X\) be an infinite-dimensional Banach space  and let
\(p_{X^*}\) be the supremal type of \(X^*\).  If \(p_{X^*} > 1\), then, for
every \(k \in \N\),
\[
        \vhelly_X(k)
        \ge
        k^{-1 + \frac{1}{p_{X^*}}}.
\]
If \(X^*\) has type \(p_{X^*}\), this matches the exponent in the upper bound
coming from the type \(p_{X^*}\) inequality for \(X^*\).
\end{thm}

\begin{proof}
Fix \(k \in \N\), \(s > k\), and \(\delta \in (0,1)\).  By
\Href{Proposition}{prp:MP}, there is a linear map
\[
        T : \ell_{p_{X^*}}^s \to X^*
\]
such that
\[
        \norm{u}_{p_{X^*}}
        \le
        \norm{Tu}_{X^*}
        \le
        (1 + \delta)\norm{u}_{p_{X^*}}
        \qquad\text{for every } u \in \ell_{p_{X^*}}^s.
\]
 Applying
\Href{Lemma}{lem:dual-subspace-quotient-helly} with
\(Z = (\ell_{p_{X^*}}^s)^*\), and then using
\Href{Lemma}{lem:lp-dual-model-helly}, we get
\[
        \vhelly_X(k)
        \ge
        \frac{1}{1 + \delta}
        \vhelly_{(\ell_{p_{X^*}}^s)^*}(k)
        \ge
          \frac{1}{1 + \delta}
        \frac{k^{-1 + \frac{1}{p_{X^*}}} - s^{-1 + \frac{1}{p_{X^*}}}}
        {1 + s^{-1 + \frac{1}{p_{X^*}}}}.
\]
Letting first \(\delta \to 0\) and then \(s \to \infty\) gives the claim.
\end{proof}

\begin{cor}
\label{cor:helly-trivial-type}
If \(X^*\) has trivial type, then, for every \(k \in \N\),
\[
        \vhelly_X(k)
        \ge
        \frac{k}{2k - 1}
        \ge
        \frac{1}{2}.
\]
In particular, \(X\) does not have the Helly approximation property.
\end{cor}

\begin{proof}
This is the endpoint construction from
\cite[Lemma 4.2]{ivanov2026banachspaceshellyapproximation}.  More precisely, if
\(X^*\) has trivial type, then for every \(\eta > 0\) and every \(k \in \N\)
there are compact convex sets
\[
        K_1,\ldots,K_{2k} \subset \ball{}_{X}
\]
such that every \(k\)-subfamily has a common point in \(\ball{}_{X}\), while
\[
        \inf_{y \in X}\max_{i \in [2k]} \dist{y}{K_i}
        \ge
        \frac{k}{2k - 1} - \eta.
\]
By the definition of \(\vhelly_X(k)\), this gives
\[
        \vhelly_X(k)
        \ge
        \frac{k}{2k - 1} - \eta.
\]
Letting \(\eta \to 0\) proves the claim.
\end{proof}

\section{Infratype upper bounds}
\label{sec:infratype-positive}

In this section we obtain new estimates that use infratype rather than
Rademacher type.  
The qualitative threshold is unchanged, because the supremal 
type and the supremal infratype coincide.  Nevertheless, infratype is the more
natural input for the deterministic sign-selection arguments below.

We begin with the Radon estimate stated in the introduction.  The proof is just
the infratype inequality applied to the differences of paired points; this is
why the Radon quantity detects infratype so directly.

\Href{Theorem}{thm:intro-infratype-radon} is a direct corollary of the last assertion of the following lemma. 
\begin{lem}
\label{lem:infratype-radon}
Let \(X\) have infratype \(p > 1\) 
with constant \(\Ip_p(X)\).  Let
\(Z_1, \dots,Z_r \subset X\) be pairwise disjoint sets, each of cardinality
\(2n\), and put \(D_j = \diam Z_j \).  Then, there is a balanced Radon split
\(S = Q_0 \sqcup Q_1\), where \(S = \bigcup_{j\in[r]} Z_j\), such that
\[
        \norm{\centroid{Q_0} - \centroid{Q_1}}_X
        \le
        \frac{\Ip_p(X)}{nr}
        \parenth{n\sum_{j \in [r]} D_j^p}^{\frac{1}{p}}.
\]
Consequently, for every \(n,r \in \N\),
\[
        \vrad_X(2n,r)
        \le
        \Ip_p(X)(nr)^{-1+ \frac{1}{p}}.
\]
\end{lem}

\begin{proof}
Pair the points of every color class arbitrarily,
\[
        Z_j=\braces{u_{j,1},v_{j,1},\ldots,u_{j,n},v_{j,n}}.
\]
Put \(y_{j,t}=u_{j,t}-v_{j,t}\).  
By the definition of infratype, we may choose signs
\(\theta_{j,t} = \pm 1\) so that
\[
        \norm{
        \sum_{j \in [r]}\sum_{t \in [n]} 
        \theta_{j,t}y_{j,t}}_X
        \le
        \Ip_p(X)
        \parenth{
        \sum_{j \in [r]}\sum_{t \in [n]}  
        \norm{y_{j,t}}_X^p}^{1/p}
        \le
        \Ip_p(X)\parenth{n
        \sum_{j \in [r]} D_j^p}^{1/p}.
\]
For each pair, put \(u_{j,t}\) in \(Q_0\) and \(v_{j,t}\) in \(Q_1\) if
\(\theta_{j,t} =1 \), and interchange the two points if \(\theta_{j,t} = - 1\).
Then, \(Q_0,Q_1\) form a balanced Radon split and
\[
        \centroid{Q_0} - \centroid{Q_1}
        =
        \frac{1}{nr}
         \sum_{j \in [r]}\sum_{t \in [n]}  \theta_{j,t}y_{j,t}.
\]
This proves the displayed centroid estimate.  Since the two centroids belong to
\(\conv Q_0\) and \(\conv Q_1\), respectively,
\[
        \dist{\conv\, Q_0}{\conv\, Q_1}
        \le
        \norm{\centroid{Q_0} - \centroid{Q_1}}_X.
\]
If \(D = \max_jD_j\), then
\[
        \frac{\dist{\conv\, Q_0}{\conv\, Q_1}}{D}
        \le
        \Ip_p(X)(nr)^{-1+ \frac{1}{p}},
\]
which gives the bound for \(\vrad_X(2n,r)\).
\end{proof}

Taking \(n=1\) and \(Z_j=\{x_j,0\}\) in the stronger estimate above recovers
the infratype inequality itself.  Thus the Radon estimate is not just a
consequence of infratype; it is essentially the same sign-selection phenomenon
written in colorful geometric language.

We next recall the averaging lemma of Kadets and Kadets 
\cite[pp. 133–134, Lemma 3]{kadets1997series}.  

\begin{prp}[No-dimemsional Colorful Carath\'eodory lemma]
\label{lem:artstein-kadets-infratype}
Let \(X\) be a Banach space of infratype \(p>1\) with constant \(\Ip_p(X)\).
Let \(A_1,\dots, A_m \subset X\) be bounded sets, and let
\(b_i\in\conv A_i\) for \(i\in[m]\).  Then, there are points \(a_i\in A_i\)
such that
\[
        \norm{\sum_{i \in[m]} a_i-\sum_{i \in [m]} b_i}_X
        \le
        \frac{2\Ip_p(X)}{2^{1-1/p}-1}
        \parenth{\sum_{i \in [m]} (\diam A_i)^p}^{1/p}.
\]
\end{prp}

For \(m\in\N\), define the colorful Carath\'eodory number \(\ccar_X(m)\) as the
infimum of all \(\rho \ge 0\) with the following property: for every choice of
bounded sets \(A_1,\ldots,A_m\subset X\), and every choice of points
\(b_i\in\conv A_i\), there are points \(a_i\in A_i\) such that
\[
        \norm{\frac1m\sum_{i=1}^m a_i-
        \frac1m\sum_{i=1}^m b_i}_X
        \le
        \rho\max_{i\in[m]}\diam A_i.
\]
The preceding proposition gives
\begin{equation}
\label{eq:ccar-infratype-bound}
        \ccar_X(m)
        \le
        K_p(X)m^{-1  + \frac{1}{p}},
        \qquad
        K_p(X):=\frac{2\Ip_p(X)}{2^{1 - \frac{1}{p}}-1}.
\end{equation}
In particular, \(\operatorname{Car}_X(m)\le\ccar_X(m)\), by taking all color
classes equal to the same set.

For the ordinary Carath\'eodory number one can also argue directly by dyadic
halving, with a slightly better constant.

\begin{lem}
\label{lem:infratype-car-direct}
If \(X\) has infratype \(p>1\), then, for every \(k\in\N\),
\[
        \operatorname{Car}_X(k)
        \le
        \frac{\Ip_p(X)}
        {2 \parenth{ 1 - 2^{-1 + \frac{1}{p}} } }
        k^{-1 + \frac{1}{p}}.
\]
\end{lem}

\begin{proof}
Let \(P\subset X\) be bounded, put \(D=\diam P\), and fix
\(a\in\conv P\).  It is enough to prove the estimate for \(a\) which is the
average of a finite multiset of points of \(P\), because such averages are dense
in \(\conv P\).  Thus, take \(N=2^s k\) points \(x_1,\ldots,x_N\in P\) whose
average is \(a\).

Suppose that at some stage we have \(2m\) points with average \(z_{2m}\).  Pair
them as \((u_i,v_i)\), \(i\in[m]\).  By the definition of infratype, choose signs so that
\[
        \norm{\sum_{i \in [m]}
        \theta_i(u_i - v_i)}_X
        \le
        \Ip_p(X)m^{\frac{1}{p}}D.
\]
Keeping one point from each pair according to these signs, we obtain \(m\)
points with average \(z_m\) satisfying
\[
        \norm{z_m-z_{2m}}_X
        \le
        \frac{\Ip_p(X)}{2}m^{-1+\frac{1}{p}} D.
\]
Iterating from \(N\) down to \(k\) gives
\[
        \dist{a}{\conv_k P}
        \le
        \frac{\Ip_p(X)}{2}
        \sum_{j = 0}^{\infty}(2^j k)^{-1+ \frac{1}{p}} D
        =
        \frac{\Ip_p(X)}{2\parenth{1-2^{-1+\frac{1}{p}}}}
        k^{-1+\frac{1}{p}}D.
\]
This proves the estimate for uniform averages with denominator \(2^s k\).
Approximating an arbitrary point of \(\conv P\) by such averages and passing
to the limit gives the same bound for every \(a\in\conv P\).
Taking the supremum over \(P\) and \(a\) proves the claim.
\end{proof}

The binary-tree proof of the colorful Tverberg theorem of 
\cite{ivanov2021no} needs a halving estimate
for color classes of arbitrary size, not only for the even sizes used in
\(\vrad_X(2n,r)\).  For odd cardinalities the extra point in each color has to
be chosen coherently.  This is precisely the role of the colorful
Carath\'eodory estimate above.

\begin{lem}[Halving for all cardinalities]
\label{prp:infratype-halving-all-n}
Let \(X\) have infratype \(p>1\) with constant \(\Ip_p(X)\).  Put
\[
        B_p(X)=K_p(X)+\Ip_p(X),
\]
where \(K_p(X)\) is defined in \eqref{eq:ccar-infratype-bound}.  Let
\(m\ge2\), let \(r\in\N\), and let \(Z_1,\ldots,Z_r\subset X\) be pairwise
disjoint sets, each of cardinality \(m\).  Put
\[
        S=\bigcup_{j\in[r]}Z_j,
        \qquad
        D=\max_{j\in[r]}\diam Z_j,
        \qquad
        d=\left\lceil\frac m2\right\rceil .
\]
Then there is a partition \(S=Q_0\sqcup Q_1\) such that, for every \(j\in[r]\),
\[
        \braces{\card{Q_0\cap Z_j},\card{Q_1\cap Z_j}}
        =
        \braces{m-d,d},
\]
and
\[
        \max_{i=0,1}
        \norm{\centroid{Q_i}-\centroid S}_X
        \le
        B_p(X)(rd)^{-1+ \frac{1}{p}}D.
\]
\end{lem}

\begin{proof}
We may translate each color class separately by minus its centroid.  This does
not change the diameters and does not change the quantities
\(\centroid{Q_i}-\centroid S\), because each admissible part takes the same
number of points from every color class.  Thus, we assume that
\(\centroid{Z_j}=0\) for every \(j\), and hence \(\centroid S=0\).

The even case $m = 2d$ follows from \Href{Lemma}{lem:infratype-radon} with a stronger than the asserted estimate.

Now suppose that \(m=2d-1\).  Since \(\centroid{Z_j}=0\), we have
\(0\in\conv Z_j\) for every \(j\).  By \eqref{eq:ccar-infratype-bound}, applied
to the sets \(Z_j\) and the points \(b_j=0\), choose one point \(w_j\in Z_j\)
for each \(j\) such that
\[
        \norm{\sum_{j \in [r]} w_j}_X
        \le
        K_p(X)r^{\frac{1}{p}}D.
\]
Set $W = \sum_{j \in [r]} w_j.$
Remove the points \(w_j\) and pair the remaining \(2d-2\) points of each color.
Writing the pair differences as \(y_{j,t}\), \(t\in[d-1]\), choose signs by
infratype so that
\[
                \norm{\sum_{j \in [r]} 
        \sum_{t \in [d-1]} 
        \theta_{j,t}y_{j,t}}_X
        \le
        \Ip_p(X)\parenth{r(d-1)}^{\frac{1}{p}}D.
\]
We denote the signed sum in the leftmost norm by $B.$
Let \(Q_1\) consist of the special points \(w_j\) and the signed choices from
all pairs; let \(Q_0=S\setminus Q_1\).  Then
\(\card{Q_1\cap Z_j}=d\) and \(\card{Q_0\cap Z_j}=d-1\).  Since the total sum
of all paired points is \(-W\), the sum of the paired points chosen for \(Q_1\)
is \((B-W)/2\).  Hence,
\[
        \sum_{x\in Q_1}x
        =
        \frac{W+B}{2}.
\]
Therefore,
\[
        \norm{\centroid{Q_1}}_X
        \le
        \frac{K_p(X) r^{\frac{1}{p}}D+
        \Ip_p(X)\parenth{r(d-1)}^{\frac{1}{p}}D}{2rd}
        \le
        \frac{B_p(X)}{2}(rd)^{-1+ \frac{1}{p}}D.
\]
Since \(\centroid S=0\),
\[
        rd\,\centroid{Q_1}+r(d-1)\,\centroid{Q_0}=0.
\]
As \(d/(d-1)\le2\), this gives
\[
        \norm{\centroid{Q_0}}_X
        \le
        B_p(X)(rd)^{-1+ \frac{1}{p}}D.
\]
The proof is complete.
\end{proof}

The remaining Carath\'eodory-type estimates are obtained by the same
combinatorial chain as in the last section of \cite{ivanov2021no}.  We include
the bounds only to keep track of the dependence on the infratype constant.

\begin{lem}[Carath\'eodory-type consequences]
\label{prp:infratype-transfer}
Let \(X\) have infratype \(p>1\).  Then, with \(K_p(X)\) and \(B_p(X)\) as
above,
\[
        \vtv_X(r,k)
        \le
        \frac{2B_p(X)}{1-2^{-1+ \frac{1}{p}}}\, 
        r^{-1+\frac{1}{p}},
\quad 
        \vsel_X(r,r^{-r})
        \le
        \parenth{\frac{2B_p(X)}{1-2^{-1+ \frac{1}{p}}}+K_p(X)}
        r^{-1+\frac{1}{p}},
\]
and, for every \(0<\eps<1\),
\[
        \vnet_X\parenth{\eps,r^r\eps^{-r}}
        \le
        \parenth{\frac{2B_p(X)}{1-2^{-1+\frac{1}{p}}}+K_p(X)}
        r^{-1+\frac{1}{p}}.
\]
\end{lem}

\begin{proof}
The Tverberg estimate is obtained by the binary-tree halving argument of
\cite[Section~5]{ivanov2021no}, using
\Href{Lemma}{prp:infratype-halving-all-n} at each split.  Along a
root-to-leaf path, the relevant color size is replaced by its ceiling half; the
rounding produces at most two terms at each dyadic scale, giving the factor
\(2/(1-2^{-1+1/p})\).  The selection estimate is the selection step from the
same section, with the additional colorful Carath\'eodory error
\(K_p(X)r^{-1+1/p}\).  The weak-net estimate follows from the usual greedy
argument applied to the selection lemma, again as in \cite[Section~5]{ivanov2021no}.
\end{proof}

We finish the section with the proof of \Href{Theorem}{thm:intro-infratype-helly}. 
We will use the following simple observation that follows from   the standard lifting argument.
\begin{lem}
\label{lem:infratype-quotient}
Let \(Q:E\to Y\) be a quotient map of norm one.  If \(E\) has infratype \(p\)
with constant \(C\), then \(Y\) has infratype \(p\) with constant at most \(C\).
\end{lem}


\begin{proof}[Proof of \Href{Theorem}{thm:intro-infratype-helly}]
Let \(\mathcal F\) be a finite family of convex subsets of \(\ball{}_{X}\) such
that every subfamily of size at most \(k\) has a common point.  For every such
subfamily choose one point in its intersection, and let \(E\subset X\) be the
finite-dimensional subspace spanned by all chosen points.  Replacing each
\(K\in\mathcal F\) by \(K\cap E\), we preserve the \(k\)-wise intersection
property and it is enough to find an approximate common point in \(E\).

The restriction map \(X^*\to E^*\) is a quotient map of norm one.  Hence, by
\Href{Lemma}{lem:infratype-quotient}, the space \(E^*\) has infratype \(p\) with
constant at most \(\Ip_p(X^*)\).  Applying
\Href{Lemma}{lem:infratype-car-direct} in \(E^*\), we get
\[
        \operatorname{Car}_{E^*}(k)
        \le
        \frac{\Ip_p(X^*)}{2\parenth{1-2^{-1+1/p}}}
        k^{-1+1/p}.
\]

We may now argue in the finite-dimensional space \(E\).  Replacing the sets by
their closures, define
\[
        f(x)=\max_{K\in\mathcal F}\dist{x}{K\cap E},
        \qquad
        \rho=\min_{x\in E}f(x),
\]
and choose a minimizer \(x_0\).  Since \(f(0)\le1\), we may assume
\(x_0\in2\ball{}_{E}\).  If \(\rho=0\), there is nothing to prove.

The standard separation condition at the minimizer gives finitely many active
sets \(K_s\in\mathcal F\) and functionals \(u_s\in E^*\), \(\norm{u_s} \le 1\), 
such that
\[
        0\in\conv\braces{u_s}
\]
and
\[
        \iprod{u_s}{y-x_0}
        \le
        -\rho
        \qquad\text{for every }y\in K_s\cap E.
\]
By the definition of \(\operatorname{Car}_{E^*}(k)\), applied to the finite set
\(\{u_s\}\) of diameter at most \(2\), we can choose
\(s_1,\ldots,s_\ell\), \(\ell\le k\), and coefficients
\(\lambda_t\ge0\), \(\sum_t\lambda_t=1\), such that
\[
        \norm{\sum_{t \in [\ell]}
        \lambda_t u_{s_t}}_{E^*}
        \le
        2\operatorname{Car}_{E^*}(k).
\]
By the \(k\)-wise intersection assumption, choose
\[
        q\in \ball{}_{E}\cap K_{s_1}\cap\cdots\cap K_{s_\ell}.
\]
Taking the corresponding convex combination of the active inequalities at
\(q\), we obtain
\[
        \rho
        \le
        -\iprod{\sum_{t=1}^\ell\lambda_tu_{s_t}}{q-x_0}
        \le
        2\operatorname{Car}_{E^*}(k)\norm{q-x_0}_E
        \le
        6\operatorname{Car}_{E^*}(k).
\]
The displayed estimate for \(\operatorname{Car}_{E^*}(k)\) gives the desired
bound.
\end{proof}

\section{Dimension strikes back}
\label{sec:dimension_strikes_back}

The preceding sections show that, in spaces of trivial type, the corresponding
no-dimensional errors need not tend to zero.  In finite-dimensional spaces this
cannot be the whole story, because every finite-dimensional Banach space has
non-trivial type $p=2$.  The point is that the dimension must then enter the estimates.
We give a few concrete examples showing that the logarithmic dependence on the
dimension in standard finite-dimensional bounds is unavoidable in the natural
regime.

For example, it was shown in \cite{ivanov2026nodimensionalresults} that
\[
        \vhelly_{\ell_1^n}(k)
        \le
        C\sqrt{\frac{\ln n}{k}}
\]
for some absolute constant \(C\).  Similarly, \cite[Theorem~3.3]{barman2015approximating}
shows that
\[
        \operatorname{Car}_{\ell_\infty^n}(k)
        \le
        C\sqrt{\frac{\ln n}{k}}
\]
for some absolute constant \(C\).  On the other hand, the classical Helly and
Carath\'eodory theorems give
\[
        \vhelly_{\ell_1^n}(n + 1)
        =
        \operatorname{Car}_{\ell_\infty^n}(n + 1)
        =
        0.
\]
Thus the relevant question is what happens between these two regimes.  The
examples below show that the logarithmic term cannot be simply removed: when
the combinatorial parameter is comparable with \(\ln n\), the corresponding
error may still be bounded from below by a positive absolute constant.
\subsection{The Helly sequence in \texorpdfstring{\(\ell_1^n\)}{ell1n}}

The following example gives a finite-dimensional obstruction for the Helly
sequence in \(\ell_1^n\) when the dimension is exponential in \(k\).

\begin{lem}
\label{lem:l1-binomial-helly}
Let \(k,s \in \N\) and \(s > k\).  Then
\[
        \vhelly_{\ell_1^{\binom{s}{k}}}(k)
        \ge
        \frac{s - k}{s + k}.
\]
In particular,
\[
        \vhelly_{\ell_1^{\binom{2k}{k}}}(k)
        \ge
        \frac{1}{3}.
\]
\end{lem}
For \(n = \binom{2k}{k}\), the quantity \(\ln n\) is bounded from above and from
below by positive absolute multiples of \(k\), while
\[
        \vhelly_{\ell_1^n}(k)
        \ge
        \frac{1}{3}.
\]
Thus, the logarithmic dependence on the dimension in the finite-dimensional
estimate cannot be removed.
\begin{proof}
Let
\[
        \Omega
        =
        \braces{A \subset [s] \st \card{A} = k}.
\]
We identify \(\ell_1^{\binom{s}{k}}\) with \(\ell_1(\Omega)\).  For each
\(i \in [s]\), define
\[
        K_i
        =
        \conv\braces{e_A \st A \in \Omega,\ i \in A}
        \subset
        \ball{}_{\ell_1(\Omega)}.
\]
If \(J \subset [s]\) and \(\card{J} = k\), then \(e_J \in K_i\) for every
\(i \in J\).  Hence every subfamily of \(\braces{K_i}_{i = 1}^s\) of size
exactly \(k\) has non-empty intersection.  The same is then true for every
subfamily of size at most \(k\), by extending it to a \(k\)-subfamily.

Assume that the \(\alpha\)-neighborhoods of all \(K_i\)'s have a common point
\(y \in \ell_1(\Omega)\).  For \(i \in [s]\), let \(f_i \in \ell_\infty(\Omega)\)
be given by
\[
        f_i(A)
        =
        \begin{cases}
        1, & i \in A,\\
        0, & i \notin A.
        \end{cases}
\]
Then \(\norm{f_i}_\infty = 1\), and \(\iprod{f_i}{x} = 1\) for every
\(x \in K_i\).  Therefore
\[
        \iprod{f_i}{y}
        \ge
        1 - \alpha
        \qquad\text{for every } i \in [s].
\]
Summing over \(i\), we get
\[
        \iprod{\sum\limits_{i \in [s]} f_i}{y}
        \ge
        s(1 - \alpha).
\]
On the other hand, for every \(A \in \Omega\), exactly \(k\) indices
\(i \in [s]\) belong to \(A\).  Hence
\[
        \sum\limits_{i \in [s]} f_i
        =
        k \mathbf 1_\Omega.
\]
Moreover, since \(y\) is \(\alpha\)-close to \(K_1 \subset \ball{}_{\ell_1(\Omega)}\),
we have \(\norm{y}_1 \le 1 + \alpha\).  Thus
\[
        s(1 - \alpha)
        \le
        \iprod{\sum\limits_{i \in [s]} f_i}{y}
        \le
        k\norm{y}_1
        \le
        k(1 + \alpha).
\]
Consequently,
\[
        \alpha
        \ge
        \frac{s - k}{s + k}.
\]
Taking \(s = 2k\) gives the second assertion.
\end{proof}

\subsection{Examples in \texorpdfstring{\(\ell_\infty^n\)}{ellinftyn}}

We now pass the \(\ell_1\)-examples to cubes.  For \(m \in \N\), define
\[
        \Phi_m : \ell_1^m \to \ell_\infty^{2^m},
        \qquad
        \Phi_m x
        =
        \parenth{\sum_{i = 1}^m \varepsilon_i x_i}_{\varepsilon \in \braces{-1,1}^m}.
\]
Then
\[
        \norm{\Phi_m x}_\infty
        =
        \sum\limits_{i \in [m]} \abs{x_i}
        =
        \norm{x}_1.
\]
Thus every lower-bound configuration in \(\ell_1^m\) gives an isometric
configuration in \(\ell_\infty^{2^m}\).

\begin{lem}[Carath\'eodory lower bounds in cubes]
\label{lem:cube-caratheodory}
Let \(m > k\).  Then
\[
        \operatorname{Car}_{\ell_\infty^{2^m}}(k)
        \ge
        1 - \frac{k}{m}.
\]
In particular,
\[
        \operatorname{Car}_{\ell_\infty^{2^{2k}}}(k)
        \ge
        \frac{1}{2}.
\]
\end{lem}

\begin{proof}
Apply the isometric embedding \(\Phi_m : \ell_1^m \to \ell_\infty^{2^m}\) to
the coordinate Carath\'eodory configuration from
\Href{Model example}{modelex:caratheodory-coordinate} with \(p = 1\).  Since
\(\Phi_m\) is an isometry, both the distance from the centroid to the
\(k\)-convex hull and the diameter of the configuration are preserved.  The
claim follows from \Href{Lemma}{lem:caratheodory-model}.
\end{proof}

For \(d = 2^{2k}\), the equality \(k = \frac{\ln d}{2\ln 2}\) holds.  Hence the
last lemma shows that the finite-dimensional Carath\'eodory estimate in
\(\ell_\infty^d\) cannot tend to zero at this scale.

\begin{lem}[Radon lower bounds in cubes]
\label{lem:cube-radon}
Let  \(n, r \in \N\).  Then
\[
        \vrad_{\ell_\infty^{2^{2nr}}}(2n,r)
        \ge
        1.
\]
\end{lem}

\begin{proof}
Apply the isometric embedding
 \(\Phi_{2nr} : \ell_1^{2nr} \to \ell_\infty^{2^{2nr}}\)
to the coordinate Radon configuration from
\Href{Model example}{modelex:radon-coordinate} with \(p = 1\).  Since the
embedding is isometric, both the distances between the convex hulls and the
diameters of the color classes are preserved.  The lower bound is therefore
\(1\).
\end{proof}

If \(d = 2^{2nr}\), then \(2nr = \frac{\ln d}{\ln 2}\).  Thus the Radon error in
\(\ell_\infty^d\) can still be at least \(1\) when the product \(nr\) is
logarithmic in the dimension.

\begin{lem}[Tverberg lower bounds in cubes]
\label{lem:cube-tverberg}
Let \(r \in \N\) and \(k \ge 2\).  Then
\[
        \vtv_{\ell_\infty^{2^{rk}}}(r,k)
        \ge
        \frac{1}{2}.
\]
\end{lem}

\begin{proof}
Apply the isometric embedding \(\Phi_{rk}:\ell_1^{rk} \to \ell_\infty^{2^{rk}}\)
to the block-coordinate Tverberg configuration from
\Href{Model example}{modelex:tverberg-coordinate} with \(p = 1\).  The embedding
preserves the relevant distances and diameters, and the model lower bound is
\(1/2\).
\end{proof}

If \(d = 2^{rk}\), then \(rk = \frac{\ln d}{\ln 2}\).  Hence the same
phenomenon occurs for colorful Tverberg: the error need not be small when the
number of sampled color classes is only logarithmic in the dimension.

\begin{lem}[Selection lower bounds in cubes]
\label{lem:cube-selection}
Let \(r \in \N\), let \(0 < \theta < 1\), and choose \(N \in \N\) such that
\[
        N \ge 2r
        \qquad\text{and}\qquad
        \frac{r}{N} < \theta .
\]
Then
\[
        \vsel_{\ell_\infty^{2^N}}(r,\theta)
        \ge
        \frac{1}{2}.
\]
\end{lem}

\begin{proof}
Apply \(\Phi_N:\ell_1^N \to \ell_\infty^{2^N}\) to the coordinate selection
configuration from \Href{Model example}{modelex:selection-coordinate} with
\(p = 1\).  Repeating the proof of \Href{Theorem}{thm:selection-sup-type} with
this isometric embedding, and using the Erd\H{o}s--Ko--Rado theorem exactly as
there, gives the claim.
\end{proof}

For fixed \(\theta\), one may choose \(N\) proportional to \(r\).  Therefore the
selection radius in a cube may be bounded below by \(1/2\) when \(r\) is
logarithmic in the dimension \(2^N\).

\section{Proofs of the theorems from the introduction}
\label{sec:intro-proofs}

\begin{proof}[Proof of \Href{Theorem}{thm:intro-qualitative}]
Assume first that \(X\) has non-trivial type.  The no-dimensional
Carath\'eodory theorem follows from Maurey's lemma.  The no-dimensional colorful Radon theorem, colorful
Tverberg theorem, selection lemma, and weak \(\eps\)-net theorem in Banach
spaces of non-trivial type were proved in \cite{ivanov2021no}.  Finally,
no-dimensional Helly theorem follows from
\cite{ivanov2026banachspaceshellyapproximation}.

Conversely, suppose that \(X\) has trivial type.  Then the no-dimensional
Carath\'eodory theorem fails by \Href{Theorem}{thm:caratheodory-sup-type}; the
colorful Radon theorem fails by \Href{Corollary}{cor:radon-trivial-type}; the
colorful Tverberg theorem fails by \Href{Corollary}{cor:tverberg-trivial-type};
and the selection lemma fails by \Href{Corollary}{cor:selection-trivial-type}.
The weak \(\eps\)-net theorem fails by \Href{Theorem}{thm:weak-net-trivial-type}.
For Helly, trivial type of \(X\) implies trivial type of \(X^*\), again by the
self-duality of non-trivial type.  Thus, the no-dimensional Helly theorem fails
by \Href{Corollary}{cor:helly-trivial-type}.  This proves all equivalences.
\end{proof}

\begin{proof}[Proof of \Href{Theorem}{thm:intro-quantitative}]
The Carath\'eodory estimate is \Href{Theorem}{thm:caratheodory-sup-type}.  The
Radon, Tverberg, and selection estimates are
\Href{Theorem}{thm:radon-sup-type}, \Href{Theorem}{thm:tverberg-sup-type}, and
\Href{Theorem}{thm:selection-sup-type}, respectively.  If \(X\) has type
\(p_X\), these lower bounds have the same powers as the upper estimates recalled
in the introduction, and hence those powers are optimal.
\end{proof}

\begin{proof}[Proof of \Href{Theorem}{thm:intro-helly-quantitative}]
This is exactly \Href{Theorem}{thm:helly-sup-type}.  If the supremal type of
\(X^*\) is attained, the lower bound from \Href{Theorem}{thm:helly-sup-type}
has the same power as the upper bound from
\cite{ivanov2026banachspaceshellyapproximation}.
\end{proof}


\end{document}